\documentclass[a4paper,12pt]{amsart}
\usepackage{amsmath, amssymb, amsthm, xcolor}
\usepackage{tikz, tikz-cd}
\usepackage{enumerate}
\usepackage[colorlinks=true, citecolor=blue, linkcolor=blue, bookmarks=true]{hyperref}
\tikzcdset{scale cd/.style={every label/.append style={scale=#1},
    cells={nodes={scale=#1}}}}

\newcounter{counter}
\newtheorem{theorem}[counter]{Theorem}
\newtheorem{lemma}[counter]{Lemma}
\newtheorem{corollary}[counter]{Corollary}
\newtheorem{proposition}[counter]{Proposition}
\newtheorem*{theorem*}{Theorem}

\theoremstyle{definition}
\newtheorem{definition}[counter]{Definition}

\newtheorem{remark}[counter]{Remark}
\newtheorem{example}[counter]{Example}

\newcommand{\B}{\mathcal{B}}
\newcommand{\Z}{\mathcal{Z}}

\newcommand{\R}{\mathcal{R}}
\newcommand{\K}{\mathbb{K}}
\newcommand{\C}{\mathbb{C}}
\newcommand{\T}{\mathbb{T}}
\newcommand{\N}{\mathbb{N}}
\newcommand{\M}{\mathbb{M}}
\newcommand{\z}{\mathbb{Z}}

\newcommand{\Hilb}{\mathrm{Hilb}}
\newcommand{\Bim}{\mathrm{Bim}}
\newcommand{\tr}{\mathrm{tr}}
\newcommand{\Tr}{\mathrm{Tr}}
\newcommand{\Ad}{\mathrm{Ad}}
\newcommand{\Aut}{\mathrm{Aut}}
\newcommand{\Inn}{\mathrm{Inn}}
\newcommand{\Out}{\mathrm{Out}}
\newcommand{\Ker}{\mathrm{Ker}}

\newcommand{\id}{\operatorname{id}}

\newcommand{\Hom}{\operatorname{Hom}}
\newcommand{\Ext}{\operatorname{Ext}}

\newcommand{\alg}{\mathrm{alg}}
\newcommand{\Aff}{\operatorname{Aff}}

\title{Anomalous symmetries of classifiable C*-algebras}

\author{Samuel Evington}
\address{\hskip-\parindent Samuel Evington, Mathematical Institute, University of M{\"u}nster, Einsteinstrasse 62, 48149 M{\"u}nster, Germany}
\email{evington@uni-muenster.de}

\author{Sergio Gir{\'o}n Pacheco}
\address{\hskip-\parindent Sergio Gir{\'o}n Pacheco, Mathematical Institute, University of Oxford, Oxford, OX2 6GG, UK.}
\email{sergio.gironpacheco@sjc.ox.ac.uk}

\thanks{{\footnotesize SE was partially supported by the Deutsche Forschungsgemeinschaft (DFG, German Research Foundation) – Project-ID 427320536 – SFB 1442, as well as under Germany's Excellence Strategy EXC 2044 390685587, Mathematics M{\"u}nster: Dynamics–Geometry–Structure, and was also partially supported by ERC Advanced Grant 834267 - AMAREC, and by EPSRC grant EP/R025061/2; SGP was supported by the Ioan and Rosemary James Scholarship awarded by St John's College and the Mathematical Institute, University of Oxford.}}

\begin{document}

\begin{abstract}
	We study the $H^3$ invariant of a group homomorphism $\phi:G \rightarrow \Out(A)$, where $A$ is a classifiable C$^*$-algebra. We show the existence of an obstruction to possible $H^3$ invariants arising from considering the unitary algebraic $K_1$ group. In particular, we prove that when $A$ is the Jiang--Su algebra $\Z$ this invariant must vanish. We deduce that the unitary fusion categories $\Hilb(G, \omega)$ for non-trivial $\omega \in H^3(G, \T)$ cannot act on $\Z$.
\end{abstract}
\maketitle
\section*{Introduction}
\renewcommand*{\thecounter}{\Alph{counter}}
The hyperfinite II$_1$ factor $\R$ of Murray and von Neumann (\cite{MvN36}) was proven to be the unique (separably acting) amenable II$_1$ factor by Connes (\cite{CO76}). Connes also classified automorphisms $\alpha \in \Aut(\R)$ up to outer conjugacy (\cite{CO75, CO77}).
Building on Connes' work, Jones classified group actions $G \rightarrow \Aut(\R)$ and group homomorphisms $G \rightarrow \Out(\R)$ for finite groups $G$ (\cite{JON80}).
These classification results were subsequently generalised by Ocneanu to the case of amenable groups (\cite{OC06}), and by Popa to the setting of quantum symmetries with his classification of subfactors $N \subset \R$ with amenable standard invariant (\cite{PO94}). 

Motivated by the recent classification of infinite-dimensional, simple, separable, unital $\rm{C}^*$-algebras of finite nuclear dimension satisfying the universal coefficient theorem (\cite{Ki95,Phi00,GLN15,EGLN15,TWW17}), we are interested in the extent to which these results on classical and quantum symmetries of $\R$ carry over to the setting of $\rm{C}^*$-algebras with the aforementioned properties (hereinafter referred to as \emph{classifiable} $\rm{C}^*$-algebras). The uniformly hyperfinite (UHF) algebras $\bigotimes_{k \in \N} \M_{n_k}$ (\cite{GlimmUHF}) and the Jiang--Su algebra $\Z$ (\cite{Jiang-Su}) are of particular interest as they can be viewed as $\rm{C}^*$-analogues of $\R$.

In this paper, we will focus on the existence question. It is known that any countable discrete group $G$ acts faithfully on any classifiable $\rm{C}^*$-algebra $A$. (Indeed, by \cite[Corollary 7.3]{Wi12} classifiable C$^*$-algebras are $\Z$-stable. Thus, using \cite[Corollary 8.8]{Jiang-Su}, we may write $A \cong A \otimes \bigotimes_{g \in G} \Z$ and define a $G$-action by permuting the tensor factors.) However, we shall show that, even for types of quantum symmetry that are very closely related to group actions, the existence question is subtle and can have both positive and negative answers. 

The symmetries that we consider are the \emph{$\omega$-anomalous} $G$-actions (\cite{JO20}), where $G$ is a group and $\omega \in Z^3(G, \T)$ is a 3-cocyle; see Definition~\ref{def:anomalous-action}. 
Whereas a $G$-action is a group homomorphism $G \rightarrow \Aut(A)$, an \emph{$\omega$-anomalous} $G$-action comes from a group homomorphism $G \rightarrow \Out(A)$ and the 3-cocyle $\omega$, or more precisely its cohomology class, can be viewed as an obstruction to lifting this group homomorphism to a group action. (We review these cohomological computations in Section~\ref{sec:H3}.)

Anomalous actions arise naturally in the classification of group actions up to outer conjugacy because the first invariant of a $G$-action $\alpha:G \rightarrow \Aut(A)$ is the normal subgroup $K = \alpha^{-1}(\Inn(A))$ and there is an induced homomorphism $G/K \rightarrow \Out(A)$. Indeed, it follows from a result of Jones  (\cite{JON79}) that every countable discrete group $G$ has an $\omega$-anomalous action on $\R$ for every 3-cocyle $\omega \in Z^3(G,\T)$. Our main results show that this can fail in the $\rm{C}^*$-setting.

\begin{theorem}\label{thm:JiangSu}
    Let $G$ be a group. Suppose there exists an $\omega$-anomalous action of $G$ on the Jiang--Su algebra $\Z$. Then $[\omega] = 0$ in $H^3(G, \T)$.   
\end{theorem}
\begin{theorem}\label{thm:UHF}
    Let $G$ be a finite group. Suppose there is an $\omega$-anomalous action of $G$ on the the UHF algebra $\bigotimes_{k \in \N} \M_{n_k}$. Let $r$ be the order of $[\omega]$ in $H^3(G, \T)$. Then $r^\infty$ divides the supernatural number $\prod_{k\in\N}n_k$.
\end{theorem}

These results are obtained by using algebraic $K_1$, which we compute with the aid of the de la Harpe--Skandalis determinant (\cite{SdlH84}). 

In addition to these no-go theorems, we show how to construct examples of $\omega$-anomalous actions of finite groups on a variety of classifiable $\rm{C}^*$-algebras. In particular, we obtain a partial converse to Theorem \ref{thm:UHF}.

\begin{theorem}\label{thm:UHF-converse}
    Let $G$ be a finite group and let $\omega \in Z^3(G, \T)$ be any 3-cocycle. There exists an $\omega$-anomalous action of $G$ on the UHF algebra $\bigotimes_{k \in \N} \M_{|G|}$.
\end{theorem}

The strategy for proving this result is to adapt Vaughan Jones' construction of $\omega$-anomalous actions of finite groups on $\R$ (\cite{JON80}) to the $\rm{C}^*$-setting. 
We make use of a technical simplification of Vaughan Jones' argument recently obtained by Corey Jones in his study of anomalous actions on commutative $\rm{C}^*$-algebras and their stabilisations (\cite{JO20}).
As a corollary of Theorem \ref{thm:UHF-converse}, we obtain $\omega$-anomalous actions of $G$ on the
Jacelon--Razak algebra $\mathcal{W}$ (\cite{Jac13}),
and $\omega$-anomalous actions of $G$ on the Cuntz algebras $\mathcal{O}_{|G|}$ and $\mathcal{O}_{2}$ (which can also be obtained using \cite{Iz93, Iz98}).

\medskip

We have chosen to present our main results in the language of anomalous actions and highlight the close connections with group actions. However, our results can also be formulated in the more general setting of \emph{$\rm{C}^*$-tensor categories} (see Section \ref{sec:UTC} for full definitions). This provides a common framework for studying group actions, anomalous group actions, and the more general notions of symmetry arising from subfactor theory  (\cite{Mug03,PO94,PIS18,LORO97,KAEV98,JOPE19}).

The model example of a $\rm{C}^*$-tensor category is $\Hilb_\C$, whose objects are finite-dimensional complex Hilbert spaces.
Loosely speaking, an action of a $\rm{C}^*$-tensor category $\mathcal{C}$ on an operator algebra $A$ is a representation of $\mathcal{C}$ as a collection of bimodules over $A$ such that tensor products in $\mathcal{C}$ correspond to tensor products of bimodules. This is formalised by the notion of a \emph{$\rm{C}^*$-tensor functor} into the $\rm{C}^*$-tensor category $\Bim(A)$, whose objects are Hilbert-$A$-bimodules (\cite{EKQR06, KW00, KWP04}).

Group actions give rise to actions of the $\rm{C}^*$-tensor category $\Hilb_\C(G)$, whose objects are $G$-graded Hilbert spaces, and anomalous actions give rise to actions of the closely related categories $\Hilb_\C(G,\omega)$. In the setting of Theorems \ref{thm:JiangSu} and \ref{thm:UHF}, we show that any action of $\Hilb_\C(G,\omega)$ comes from an anomalous action of $G$. We obtain the following corollaries.
\begin{corollary}\label{cor:UTC-JiangSu}
    Let $G$ be a group. There exists a $\rm{C}^*$-tensor functor $\Hilb_\C(G,\omega) \rightarrow \Bim(\Z)$ if and only if $[\omega] = 0$ in $H^3(G, \T).$
\end{corollary}
\begin{corollary}\label{cor:UTC-UHF}
    Let $G$ be a finite group. There exists a $\rm{C}^*$-tensor functor $\Hilb_\C(G,\omega) \rightarrow \Bim(\bigotimes_{k \in \N} \M_{n_k})$ only if, letting $r$ be the order of $[\omega]$ in $H^3(G, \T)$, $r^\infty$ divides the supernatural number $\prod_{k\in\N}n_k$.
\end{corollary}

These results show that there are additional obstructions to the existence of actions of C$^*$-tensor categories on C$^*$-algebras beyond restrictions on the dimensions of simple objects (\cite{KWP04}).

In the case of Theorem \ref{thm:UHF-converse}, the existence of a $\rm{C}^*$-tensor functor $\Hilb_\C(G,\omega) \rightarrow \Bim(\bigotimes_{k \in \N} \M_{|G|})$ is an immediate corollary. In fact, in the general setting of $\rm{C}^*$-tensor categories an alternative strategy for proving Theorem \ref{thm:UHF-converse} becomes apparent, based on an adaptation of the Ocneanu compactness argument to the $\rm{C}^*$-setting (\cite{OC88}).

\subsection*{Structure of the paper}
In Section \ref{sec:prelims}, we fix our notation for the paper and review preliminary material on group cohomology and algebraic $K_1$. In Section \ref{sec:anomalous-actions}, we recall the definition of anomalous actions and the related notion of a $G$-kernel. In Section \ref{sec:obstructions}, we prove Theorems \ref{thm:JiangSu} and \ref{thm:UHF}. In Section \ref{sec:existence}, we discuss existence results for anomalous actions and prove Theorem \ref{thm:UHF-converse}. In Section \ref{sec:UTC}, we review C$^*$-tensor categories and prove Corollaries \ref{cor:UTC-JiangSu} and \ref{cor:UTC-UHF}.

\subsection*{Acknowledgements} This paper is based on research started during the OSU Summer Research Program on Quantum Symmetries (supported by National Science Foundation grant DMS 1654159). The authors would also like to thank Roberto Hernandez Palomares, Corey Jones, David Penneys and Stuart White for contributing to discussions that have informed this paper.

\numberwithin{counter}{section}
\section{Preliminaries}\label{sec:prelims}

\subsection{Notation}
For a non-unital $\rm{C}^*$-algebra $A$, we write $A^\sim$ for the minimal unitisation and $M(A)$ for the multiplier algebra. 

Let $A$ be a unital $\rm{C}^*$-algebra. 
We write $U(A)$ for the unitary group of $A$, and we write $\Aut(A)$ for the group of automorphisms of $A$.
Given $u \in U(A)$, we write $\Ad(u)$ for the inner automorphism $a \mapsto uau^*$.
The normal subgroup of $\Aut(A)$ consisting of all inner automorphisms is denoted $\Inn(A)$, and we write $\Out(A)$ for the quotient $\Aut(A)/\Inn(A)$.  

We write $K_0(A)$ and $K_1(A)$ for the topological $K$-groups of $A$. 
We denote the tracial state space by $T(A)$, and we write $\Aff(T(A))$ for the space of continuous affine functions $T(A) \rightarrow \mathbb{R}$.
We recall that the \emph{pairing map} $\rho_A:K_0(A) \rightarrow \Aff(T(A))$ is characterised by $\rho_A([e])(\tau) = (\tau \otimes \Tr_n)(e)$ for projections $e \in A \otimes \M_n$. 

We set $U_\infty(A) = \bigcup_{n \in \N} U_n(A)$, where $U_n(A)$ denotes the unitary group of $\M_n(A)$ and 
we view $U_n(A) \subseteq U_{n+1}(A)$ via the embedding $u \mapsto u \oplus 1_A$. 
We endow $U_\infty(A)$ with the direct limit topology, i.e.\ $V \subseteq U_\infty(A)$ is open if and only if $V \cap U_n(A) \subseteq U_n(A)$ is open for all $n \in \N$. 

Strictly speaking, $U_\infty(A)$ endowed with this topology may not be a topological group as multiplication may not be jointly continuous (see \cite{TSH98}). 
However, this naive choice of topology will suffice for our purposes because every compact subset $K \subseteq U_\infty(A)$ lies in $U_N(A)$ for some $N \in \N$ by \cite[Lemma 1.7]{Gl05}. In particular, continuous paths $[0,1] \rightarrow U_\infty(A)$ and homotopies $[0,1]^2 \rightarrow U_\infty(A)$ will factor through $U_N(A)$ for some $N \in \N$. Consequently, $K_1(A)$ can be identified with the space of path components of $U_\infty(A)$, and $K_0(A)$ can be identified with the fundamental group of $U_\infty(A)$ with the identity as the basepoint, using Bott periodicity.

We write $U_n^{(0)}(A)$ for the path component of the identity in $U_n(A)$ for $n \in \N \cup \{\infty\}$. Note that  $U_\infty^{(0)}(A) = \bigcup_{n\in\N} U_n^{(0)}(A)$.

\subsection{Group cohomology}
We recall the definition of the cohomology groups $H^n(G, M)$ where $G$ is a group and $M$ is a $\mathbb{Z}G$-module. Further information can be found in \cite{BR12}.

Fix $G$ and $M$. An \emph{$n$-cochain} is a function $\phi:G^n \rightarrow M$. The set of all $n$-cochains $C^n(G,M)$ inherits a $\mathbb{Z}G$-module structure from $M$.
The coboundary maps $d^{n+1}:C^n(G,M) \rightarrow C^{n+1}(G,M)$ are defined by
\begin{align}\label{eqn:cohomology-def}
    d^{n+1}(\phi)(g_1,g_2,\ldots,g_{n+1}) &=
    g_1\phi(g_2,\ldots,g_{n+1}) \\
    &+ \sum_{i=1}^{n} (-1)^i\phi(g_1,\ldots,g_ig_{i+1},\ldots, g_{n+1}) \nonumber\\
    &+ (-1)^{n+1}\phi(g_1,\ldots, g_n). \nonumber
\end{align}
An \emph{n-cocycle} is an $n$-cochain $\phi$ satisfying $d^{n+1}\phi = 0$; the set of all $n$-cocyles form an abelian group $Z^n(G,M)$ under addition. An \emph{n-coboundary} is an $n$-cochain $\phi$ satisfying $\phi = d^{n}\psi$ for some $(n-1)$-cochain $\phi$; the set of all $n$-coboundaries forms an abelian group $B^n(G,M)$ under addition.

Since $d^{n+1} \circ d^n = 0$, it follows that $B^n(G,M) \subseteq Z^n(G,M)$. The quotient group $H^n(G,M) = Z^n(G,M)/B^n(G,M)$ is the $n$-th \emph{cohomology group} of $G$ with \emph{coefficients} in $M$.

We warn the reader that the formula \eqref{eqn:cohomology-def} assumes that $M$ is written additively. In multiplicative notation, the right hand side would be 
\begin{equation*}\label{eqn:cohomology-def-2}
    \alpha_{g_1}(\phi(g_2,\ldots,g_{n+1}))\cdot \prod_{i=1}^{n} \phi(g_1,\ldots,g_ig_{i+1},\ldots, g_{n+1})^{\epsilon_i} \cdot \phi(g_1,\ldots, g_n)^{\epsilon_{n+1}}, \nonumber
\end{equation*}
where $\epsilon_i = (-1)^i$ and where $\alpha_g$ denotes the action of $g$ on $M$.

In this paper, $M$ will often be an abelian group endowed with a trivial $\mathbb{Z}G$-module structure, where $gm = m$  (additive notation) or $\alpha_g(m) = m$ (multiplicative notation) for all $g \in G$ and $m \in M$. However, for some of our more general results, we shall need coefficient modules with non-trivial $G$-actions.

\subsection{The de la Harpe--Skandalis determinants}\label{sec:SdlH}

Let $A$ be a unital $\rm{C}^*$-algebra. Given a tracial state $\tau \in T(A)$, de la Harpe and Skandalis defined a group homomorphism 
\begin{equation*}
    \Delta_\tau: U_\infty^{(0)}(A) \rightarrow \frac{\mathbb{R}}{\tau_*(K_0(A))},
\end{equation*}
where $\tau_*$ is the induced state on $K_0(A)$. We outline the construction below, referring the reader to \cite{SdlH84} or \cite{dlH13} for the full details.

Suppose $u \in U^{(0)}_n(A)$. Let $\xi:[0,1]\rightarrow U_n(A)$ be a smooth path with $\xi(0) = 1_A$ and $\xi(1) = u$.\footnote{By \cite[Proposition 2.1.6]{IntroK} for example, $u$ can be written as a product of exponentials, so there exists a path of the form $\xi(t) = \exp(ith_1)\cdots\exp(ith_r)$, which is clearly $C^\infty$-smooth.} Then 
\begin{align}\label{eqn:SdlH-def}
    \Delta_\tau(u) = \frac{1}{2\pi i}\int_0^1 (\tau \otimes \Tr_{n})(\xi'(t)\xi(t)^{-1}) \, \mathrm{d}t + \tau_*(K_0(A)).
\end{align}
It is not immediately clear that $\Delta_\tau$ is independent of the choice of path $\xi$. However, de la Harpe and Skandalis prove this in \cite[Lemme 1]{SdlH84}. The main ingredient in their argument is the Bott periodicity theorem (see for example \cite[Section 9]{Bl98}), which says that up to homotopy loops in $U_\infty(A)$ come from elements of $K_0(A)$. A direct computation using the product rule
\begin{align}
    (\xi_1\xi_2)'(t) = \xi_1'(t)\xi_2(t) + \xi_1(t)\xi_2'(t)
\end{align}
and the trace identity for $\tau$, shows that $\Delta_\tau$ is a group homomorphism.
We draw special attention to the case where $u = \exp(2\pi ih)$ for some self-adjoint $h \in M_n(A)$. 
Taking $\xi$ to be the path $\xi(t) = \exp(2\pi ith)$, direct computation gives $\Delta_\tau(u) = \tau(h) + \tau_*(K_0(A))$. 
In particular, for $x \in \mathbb{R}$, we have 
\begin{equation}\label{eqn:SdlH-scalars}
    \Delta_\tau(e^{2\pi i x} 1_{A}) = x + \tau_*(K_0(A)).    
\end{equation}

 Combining the de la Harpe--Skandalis determinants $\Delta_\tau$ for all $\tau \in T(A)$, we obtain a group homomorphism
    \begin{equation}
        \Delta_A:U^{(0)}_\infty(A) \rightarrow \frac{\Aff(T(A))}{\rho_A(K_0(A))},
    \end{equation}
where $\Aff(T(A))$ denotes the real-valued, continuous affine functions on the trace space of $A$ and $\rho_A$ is the paring map. We call this map the \emph{universal} de la Harpe--Skandalis determinant.

\begin{remark}
In \cite{SdlH84}, de la Harpe and Skandalis carry out their construction with any continuous linear map $\tau:A \rightarrow E$ into a Banach space $E$ satisfying the trace identity $\tau(ab) = \tau(ba)$. The universal de la Harpe--Skandalis determinant can also be obtained by this method starting with the universal trace $\Tr:A \rightarrow \Aff_\C(T(A))$. By \cite[Proposition 2.7]{CUPE79}, we have $\Ker(\Tr) = \overline{\mathrm{span}}\{ab-ba: a,b \in A\}$, and the universal trace can alternatively be viewed as taking the quotient of $A$ by this closed subspace. 
\end{remark}

\subsection{Unitary algebraic \texorpdfstring{$K_1$}{K1}} \label{sec:algebraic-K1}

Let $A$ be a unital $\rm{C}^*$-algebra. Then the \emph{unitary algebraic $K_1$-group} of $A$ is defined as the abelianisation of $U_\infty(A)$, i.e.\
\begin{align}
    K_1^{\alg} (A) = \frac{U_\infty(A)}{[U_\infty(A),U_\infty(A)]}.
\end{align}

A couple of variants of the algebraic $K_1$-group are possible. Firstly, one can replace unitary groups with general linear groups throughout (see for example \cite{HIG88}). Secondly, one can define \emph{Hausdorffised} unitary algebraic $K_1$ by replacing the commutator subgroup $[U_\infty(A),U_\infty(A)]$ by its closure in the direct limit topology on $U_\infty(A)$ (see for example \cite{Th95}). We will not use these variants in our main results but may refer to them in some remarks.

There is a canonical surjection $\kappa_A: K_1^{\alg}(A) \twoheadrightarrow K_1(A)$, which is typically not injective. For example, $K_1(\C) = 0$ as the unitary groups $U_n(\C)$ are path connected, whereas  $K_1^{\alg}(\C) \twoheadrightarrow \T$ as the determinant is well-defined on $K_1^{\alg}(\C)$-classes. (In fact, by an elementary diagonalisation argument, every unitary matrix with determinant one is a product of commutators, so $K_1^{\alg}(\C) \cong \T$.)

More generally, de la Harpe--Skandalis determinants can be used to extract some information about the kernel of the map $\kappa_A$. Using a result of Ng and Robert (\cite{KernelDet}), the kernel of $\kappa_A$ is completely determined by this information when $A$ is a classifiable $\rm{C}^*$-algebra.

\begin{theorem}[{cf.\ \cite[Theorem 1.1]{KernelDet}}]\label{thm:classifiable-K1-alg}
    Let $A$ be a unital, simple, separable, exact and $\Z$-stable $\rm{C}^*$-algebra with $T(A) \neq \emptyset$. There is a short exact sequence 
    \begin{equation}\label{eqn:alg-K1-ses}
        0 \rightarrow \frac{\Aff(T(A))}{\rho_A(K_0(A))} \xrightarrow{(\bar{\Delta}_A)^{-1}} K^{\alg}_1(A) \xrightarrow{\kappa_A} K_1(A) \rightarrow 0. 
    \end{equation}
\end{theorem}
\begin{proof}
    The $K_1$-class of a unitary $u \in U_\infty(A)$ is precisely its path component in $U_\infty(A)$. 
Moreover, for any $u,v \in U_N(A)$, we have
\begin{equation}
    [u,v] \oplus 1_A \oplus 1_A = [u \oplus u^* \oplus 1_A, v \oplus 1_A \oplus v^*].
\end{equation}
Therefore,
    \begin{equation}
        \Ker(\kappa_A) = \frac{U^{(0)}_\infty(A)}{[U^{(0)}_\infty(A),U^{(0)}_\infty(A)]}.
    \end{equation}
It remains to show that the universal de la Harpe--Skandalis determinant 
    \begin{equation}
        \Delta_A:U^{(0)}_\infty(A) \rightarrow \frac{\Aff(T(A))}{\rho_A(K_0(A))}
    \end{equation}
is surjective and that its kernel is precisely $[U^{(0)}_\infty(A),U^{(0)}_\infty(A)]$.
    
For every $f \in \Aff(T(A))$, there is a self-adjoint $h \in A$ with $\tau(h) = f(\tau)$ for all $\tau \in T(A)$ by \cite[Theorem 9.3]{Lin07}, which builds on results of Cuntz and Pedersen (\cite{CUPE79}). Then $\Delta_A(\exp(2\pi ih)) = f + \rho(K_0(A))$. Therefore, $\Delta_A$ is surjective.  
    
By construction $\Ker(\Delta_A) \supseteq [U^{(0)}_\infty(A),U^{(0)}_\infty(A)]$. The reverse inclusion follows from \cite[Theorem 1.1]{KernelDet} since $A$ is a simple, separable, pure C$^*$-algebra with stable rank one and where all quasitraces are traces. (Pureness is a consequences of $\Z$-stability; see \cite[Proposition 2.7]{Wi12}. Stable rank one also follows from $\Z$-stability when $A$ is stably finite; see \cite[Theorem 6.7]{RO04}. All quasitraces are traces as $A$ is exact; see \cite{HAA14}.)
\end{proof}

\begin{remark}
    By \cite[Theorem 3]{SdlH84a}, every unitary in a unital, simple, \emph{infinite} C$^*$-algebra in the path component of the identity is a product of commutators. Hence, the canonical surjection $\kappa_A:K_1^{\alg}(A) \twoheadrightarrow K_1(A)$ is an isomorphism in this case. Together with Theorem \ref{thm:classifiable-K1-alg} this facilitates the computation of $K_1^{\alg}(A)$ for any classifiable C$^*$-algebra. 
\end{remark}

Given a unital $^*$-homomorphism $f:A\rightarrow B$. There is a well defined group homomorphism $K_1^{\alg}(f):K_1^{\alg}(A) \rightarrow K_1^{\alg}(B)$ given by $[u]_{K_1^{\alg}(A)} \mapsto [f(u)]_{K_1^{\alg}(B)}$. In the language of category theory, $K_1^{\alg}(\cdot)$ is a covariant functor from the category of unital C$^*$-algebras to the category of abelian groups. The same is true for $K_1(\cdot)$ and $\Aff(T(\cdot))$.

The short exact sequence \eqref{eqn:alg-K1-ses} is natural in the sense that a morphism of unital $\rm{C}^*$-algebra will induce a morphism between the corresponding short exact sequences. For every $A$, the short exact sequence \eqref{eqn:alg-K1-ses} will split, since $\Aff(T(A))$ is a divisible group. However, the splitting is not natural.  

\section{\texorpdfstring{$G$}{G}-Kernels and anomalous actions} \label{sec:anomalous-actions}

Let $A$ be a unital $\rm{C}^*$-algebra and let $G$ be a group. A \emph{$G$-kernel} is a group homomorphism $G \rightarrow \Out(A)$. Motivated by an analogous concept of the same name in group theory (\cite{EILMAC47}), the study of $G$-kernels in the setting of von Neumann algebras was initiated by Nakamura and Takeda (\cite{NATA59,Ta59}) and developed by Sutherland (\cite{Su80}) and Jones (\cite{JON80}).

In this section, we review the definition of the $H^3$ invariant of a $G$-kernel and the definition of an anomalous action.

\subsection{The \texorpdfstring{$H^3$}{H3} invariant}\label{sec:H3}
Let $A$ be a unital $\rm{C}^*$-algebra and let $G$ be a group. Fix a $G$-kernel $\bar{\theta}:G \rightarrow \Out(A)$. 

For each $g \in G$, choose a lift $\theta_g \in \Aut(A)$ for $\bar{\theta}(g)$.
Since $\bar{\theta}$ is a group homomorphism, there exist unitaries $u_{g,h} \in U(A)$ such that $\theta_{gh} = \Ad(u_{g,h})\theta_g\theta_h $ for each $g, h \in G$. 
Given $g,h,k\in G$, we may compute  $\theta_{ghk}$ in two different ways:
\begin{align}
    \theta_{ghk} &= \Ad(u_{gh,k})\theta_{gh}\theta_k\\
    &= \Ad(u_{gh,k})\Ad(u_{g,h})\theta_{g}\theta_{h}\theta_k\nonumber\\
    &= \Ad(u_{gh,k}u_{g,h})\theta_{g}\theta_{h}\theta_k,\nonumber\\
    \nonumber \\
    \theta_{ghk} &= \Ad(u_{g,hk})\theta_g\theta_{hk}\\
    &= \Ad(u_{g,hk})\theta_g\Ad(u_{h,k})\theta_{h}\theta_{k} \nonumber\\
    &= \Ad(u_{g,hk})\Ad(\theta_g(u_{h,k}))\theta_g\theta_{h}\theta_{k} \nonumber\\
    &= \Ad(u_{g,hk}\theta_g(u_{h,k}))\theta_g\theta_{h}\theta_{k}. \nonumber
\end{align}
Hence, $\Ad(u_{gh,k}u_{g,h}) = \Ad(u_{g,hk}\theta_g(u_{h,k}))$. 

The kernel of the group homomorphism $\Ad:U(A) \rightarrow \Aut(A)$ is the centre of the unitary group $Z(U(A))$. Therefore, we may define a function $\omega:G^3 \rightarrow Z(U(A))$ by
\begin{equation}\label{eqn:def-omega}
\omega(g,h,k) = u_{g,hk}\theta_g(u_{h,k})u_{g,h}^{-1}u_{gh,k}^{-1}
\end{equation}
for all $g,h,k \in G$.

The group $Z(U(A))$ is abelian, and can be endowed with a $\mathbb{Z}G$-module structure where $g$ acts via $\theta_g|_{Z(U(A))}$. So $\omega$ is a 3-cochain with coefficients in $Z(U(A))$. Moreover, a simple but tedious computation shows that $d\omega = 0$; see  \cite[Lemma 7.1]{EILMAC47}. Hence, $\omega \in Z^3(G,Z(U(A)))$. The cohomology class $[\omega] \in H^3(G, Z(U(A))$ is the \emph{$H^3$ invariant} of the $G$-kernel $\bar{\theta}$. It does not depend on the choice of lifts $\theta_g$ or the choice of unitaries $u_{g,h}$; see \cite[Theorem 7.1]{EILMAC47}.

\subsection{Anomalous actions}\label{sec:omega-anomalous}

We now recall the definition of an anomalous action from \cite[Definition 1.1]{JO20}. 

\begin{definition}\label{def:anomalous-action}
    Let $G$ be a group and let $\omega \in Z^3(G, \T)$. An \emph{$\omega$-anomalous action} of $G$ on a $\rm{C}^*$-algebra $A$ consists of an automorphism $\theta_g$ for every $g \in G$ and unitaries $u_{g,h} \in U(M(A))$ such that
    \begin{enumerate}
        \item $\theta_{gh} = \Ad(u_{g,h})\theta_g\theta_h$ for $g, h \in G$,\label{cond:1}
        \item $\omega(g,h,k)u_{gh,k}u_{g,h} = u_{g,hk}\theta_g(u_{h,k})$ for $g,h,k \in G$.\label{cond:2}
    \end{enumerate}
\end{definition}

Formally, the data for an $\omega$-anomalous action is a pair of functions $(\theta, u)$ where $\theta:G \rightarrow \Aut(A)$ and $u:G \times G \rightarrow U(M(A))$. However, we shall typically use the subscript notation $\theta(g) = \theta_g$ and $u(g,h) = u_{g,h}$ to improve readability.

The relationship between anomalous actions of $G$ and $G$-kernels is straightforward when $A$ is a unital and has trivial centre. 
\begin{remark}\label{rem:changing-cocycle}
Let $A$ be a unital $\rm{C}^*$-algebra with $Z(A) = \C$.
\begin{itemize}
    \item  Every $G$-kernel lifts to an $\omega$-anomalous action of $G$ on $A$, where the cohomology class of $\omega$ coincides with the $H^3$ invariant. This follows from the derivation of the $H^3$ invariant in Section \ref{sec:H3} as $Z(U(A)) = U(Z(A)) = \T$.
    \item  Suppose $(\theta, u)$ is an $\omega$-anomalous actions of the group $G$ on a $\rm{C}^*$-algebra and $\omega'$ is cohomologous to $\omega$. Then $\omega' = d\lambda\cdot\omega$ for some 2-cochain $\lambda \in C^2(G,\T)$. Setting $u'_{g,h} = \lambda(g,h)u_{g,h}$, we have that $(\theta, u')$ is an $\omega'$-anomalous action of $G$ on $A$. 
    \item Conversely, every $\omega$-anomalous action of $G$ induces a $G$-kernel with $H^3$-invariant $[\omega]$.
\end{itemize}
\end{remark}

Suppose $A$ is a  unital C$^*$-algebra with $Z(A) \neq \C$. Then lifting a $G$-kernel gives a cocycle $\omega \in Z^3(G, Z(U(A)))$. In \cite{JO20}, the decision to restrict to cocycles with coefficients in $\T \subseteq Z(U(A))$ is justified by physical interpretations. While we are mostly interested in simple $\rm{C}^*$-algebras, which necessarily have $Z(A) = \C$, our methods don't fundamentally require $A$ to have trivial centre, so could be extended to work with a more general notion of anomalous actions.

\subsection{Anomalous Actions on \texorpdfstring{$\R$}{R}}\label{sec:actions-on-R}

In \cite{CO77}, Connes considers automorphisms $\alpha \in \Aut(\R)$ such that $\alpha^n = \Ad(u)$ for some $u \in U(\R)$. He shows that $\alpha(u) = \gamma u$ for some $n$-th root of unity $\gamma \in \C$. 

These computations can be viewed as a special case of those of Section \ref{sec:H3}, when $G$ is the cyclic group $\mathbb{Z}_n$. In particular, the cohomology group $H^3(\mathbb{Z}_n,\T)$ is cyclic of order $n$ and can be identified with the group of $n$-roots of unity; see for example \cite[Proposition 2.3]{HU14}.

Moreover, Connes constructs an automorphism $s_n^\gamma$ with order $n$ in $\Out(\R)$ and such that $s_n^\gamma(u) = \gamma u$ for every $n$-th root of unity $\gamma$. We record this result of Connes below.
\begin{theorem}[{cf. \cite[Proposition 1.6]{CO77}}]\label{thm:Connes}
    Fix $n\in\N$. View $\R = \overline{\bigotimes}_{i\in \N} (\M_n, \tr_n)$. Let $\pi_i:\M_n \rightarrow \R$ be the embedding into the $i$-th tensor factor, and let $\theta:\R \rightarrow \R$ be the endomorphism such that $\theta\pi_{i} = \pi_{i+1}$ for all $i \in \N$. 
    
    Let $\gamma$ be an $n$-th root of unity.
    Set 
    \begin{align}
        u &= \sum_{j=1}^n \gamma^j\pi_1(e_{j,j})\\
        v &= \pi_1(e_{n,1})\theta(u) + \sum_{j=1}^{n-1} \pi_1(e_{j,j+1}).
    \end{align}
    Then the sequence $(\Ad(v\theta(v)\theta^2(v)\cdots\theta^k(v)))_{k=1}^\infty$ converges pointwise in the $\|\cdot\|_2$-norm topology to an automorphism $s_n^\gamma$ such that $(s_n^\gamma)^n = \Ad(u)$ and $s_n^\gamma(u) = \gamma u$.
\end{theorem}
\begin{remark}{(cf. \cite[Proposition 1.6]{CO77})}
    For the purposes of Section~\ref{sec:existence}, we wish to draw attention to one aspect of Connes proof.

    For any element of the algebraic tensor product $\bigodot_{i\in \N} \M_n$, Connes shows that the sequence $(\Ad(v\theta(v)\theta^2(v)\cdots\theta^k(v))(x))_{k=1}^\infty$ is eventually constant. It follows that $s_n^\gamma$ restricts to an automorphism of the UHF algebra $\bigotimes_{i\in \N} \M_n$ and $(\Ad(v\theta(v)\theta^2(v)\cdots\theta^k(v)))_{k=1}^\infty$ converges in the point norm topology on this subalgebra.
\end{remark}

Given Connes' automorphism $s_n^\gamma$, we can define a $\mathbb{Z}_n$-kernel $\bar{\theta}$ on $\R$ by $\bar{\theta}(i) = (s_n^\gamma)^i + \Inn(\R)$. As $\gamma$ ranges over all $n$-th roots of unity, all possible $H^3$ invariants are realised. As outlined in Remark \ref{rem:changing-cocycle}, it follows that for any $\omega \in Z^3(\mathbb{Z}_n,\T)$, there is an $\omega$-anomalous $\mathbb{Z}_n$-action on $\R$. 

Connes result was later generalised by Jones to cover all countable discrete groups. We record this result of Jones below, translated into the language of anomalous actions using Remark \ref{rem:changing-cocycle}.
\begin{theorem}[{cf. \cite[Theorem 2.5]{JON79}}]
    Let $G$ be a countable discrete group and $\omega \in Z^3(G,\T)$. Then there exists a $\omega$-anomalous action of $G$ on $\R$. 
\end{theorem}

In Jones' construction, $\R$ is realised as a (twisted) crossed product. We will revisit this construction in Section \ref{sec:existence}. 

\section{\texorpdfstring{$\rm{C}^*$}{C*}-Obstructions}\label{sec:obstructions}

In this section, we showcase the obstruction to the existence of $\omega$-anomalous actions on $\rm{C}^*$-algebras, which we will then use to prove Theorem \ref{thm:JiangSu} and Theorem \ref{thm:UHF}. 
The obstruction arises as the unitary group of a C$^*$-algebra can have non-trivial abelian quotients. This is not the case for $\R$, as every unitary in $\R$ can be written as a product of commutators (see for example \cite{BR67}).

We begin by isolating the key computation as a proposition. 
\begin{proposition}\label{HomtoAb}
    Let $G$ be a group and $\omega \in Z^3(G,\T)$.
	Let $(\theta, u)$ be an $\omega$-anomalous action of $G$ on a unital $\rm{C}^*$-algebra $A$. 
	Suppose $q:U(A) \rightarrow M$ is a group homomorphism into some $\mathbb{Z}G$-module $M$, such that $q(\theta_g(v)) = g \cdot q(v)$ for all $v \in U(A)$ and $g\in G$. Then $[q \circ \omega]=0$ in $H^3(G, M)$.
\end{proposition}
\begin{proof}
Let $g,h,k\in G$. Since addition in $M$ is commutative, applying $q$ to \eqref{eqn:def-omega} yields
\begin{align}
q(\omega(g,h,k)1_A) &= q(u_{g,hk})+q(\theta_g(u_{h,k}))-q(u_{g,h})-q(u_{gh,k})\\
&= g \cdot q(u_{h,k}) -q(u_{gh,k}) + q(u_{g,hk}) - q(u_{g,h})\nonumber \\
&= d\eta(g,h,k)\nonumber 
\end{align}
where $\eta$ is the $2$-cochain defined by $\eta(g,h) = q(u_{g,h})$.
\end{proof}

In order to make use of Proposition \ref{HomtoAb}, we need a candidate for the homomorphism $q$. This is where the unitary algebraic $K_1$ group enters the picture (see Section \ref{sec:algebraic-K1}). By construction, $K^{\alg}_1(A)$ is an abelian group and every automorphism of $A$ induces an automorphism of $K^{\alg}_1(A)$ with inner automorphisms acting trivially. Hence, an anomalous action on $A$ gives rise to a $\mathbb{Z}G$-module structure on $K^{\alg}_1(A)$.

The reason for working with $K^{\alg}_1(A)$ instead of $K_1(A)$ is that scalars $\lambda 1_A$ always have trivial $K_1$ class but can have non-trivial $K_1^{\alg}$ class. This is necessary in order for the conclusion of Proposition \ref{HomtoAb} to be non-trivial when $Z(A) = \C$. In the case of the Jiang-Su algebra $\Z$ it would be sufficient to use the Hausdorffised version of algebraic $K_1$. However, for UHF algebra the non-Hausdorffised version of unitary algebraic $K_1$ is required.

\begin{theorem}\label{thm:K1-obstruction}
    Let $G$ be a group and $\omega \in Z^3(G,\T)$.
	Let $(\theta, u)$ be an $\omega$-anomalous action of $G$ on a unital $\rm{C}^*$-algebra $A$. 
	View $K^{\alg}_1(A)$ as a $\mathbb{Z}G$-module where $g$ acts via $K^{\alg}_1(\theta_g)$.
	Then the 3-cocycle given by  $(g,h,k) \mapsto [\omega(g,h,k)1_A]_{K^{\alg}_1}$ is trivial in $H^3(G, K^{\alg}_1(A))$.
\end{theorem}
\begin{proof}
    The result follows from Proposition \ref{HomtoAb}, taking $M = K^{\alg}_1(A)$ with the induced $\mathbb{Z}G$-module structure and $q:U(A) \rightarrow K^{\alg}_1(A)$ to be the map $u \mapsto [u]_{K^{\alg}_1}$.
\end{proof}

Theorem \ref{thm:classifiable-K1-alg} can compute the unitary algebraic $K_1$-groups of classifiable $\rm{C}^*$-algebras, and extract information about the $\mathbb{Z}G$-action. The following special case can be deduced from Theorem \ref{thm:K1-obstruction}; however, we choose to derive it directly from Proposition \ref{HomtoAb}. 

\begin{corollary}\label{cor:invariant-trace}
    Let $G$ be a group and $\omega \in Z^3(G,\T)$. 
    Let $A$ be a unital $\rm{C}^*$-algebra with $K_1(A) = 0$.
	Let $(\theta, u)$ be an $\omega$-anomalous action of $G$ on $A$. 
	Let $\tau \in T(A)$ be invariant under $\theta_g$ for all $g \in G$.
	Then $[\Delta_\tau \circ \omega]=0$ in $H^3(G, \mathbb{R}/ \tau_*(K_0(A)))$, where
\begin{equation}
    \Delta_\tau: U_\infty^{(0)}(A) \rightarrow \frac{\mathbb{R}}{\tau_*(K_0(A))}
\end{equation}
is the de la Harpe--Skandalis determinant with respect to $\tau$, and the abelian group $\mathbb{R}/ \tau_*(K_0(A))$ has the trivial $\mathbb{Z}G$-module structure.
\end{corollary}
\begin{proof}
    Since $K_1(A) = 0$, we have $U_\infty(A) = U_\infty^{(0)}(A)$. Therefore, we can apply Proposition \ref{HomtoAb} with $q = \Delta_\tau$. The fact that $\Delta_\tau(\theta_g(v)) = \Delta_\tau(v)$ follows from \eqref{eqn:SdlH-def} since $\tau$ is invariant under $\theta_g$.
\end{proof}

We are now ready to prove Theorem \ref{thm:JiangSu}. 
\begin{proof}[Proof of Theorem \ref{thm:JiangSu}]
   The Jiang--Su algebra $\Z$ has a unique tracial state $\tau$. Moreover,  $K_0(\Z) \cong \mathbb{Z}$ with the isomorphism given by $\tau_*$, and $K_1(\Z) = 0$;  see \cite[Theorem 1]{Jiang-Su}.  As $\tau$ is the unique tracial state, it is invariant under all automorphisms.
    
    Suppose there exists an $\omega$-anomalous action of $G$ on $\Z$.
    Then, by Corollary \ref{cor:invariant-trace}, we have  $[\Delta_\tau \circ \omega] = 0$ in $H^3(G, \mathbb{R}/\mathbb{Z})$.  However, $\Delta_\tau$ restricted to $Z(U(\Z)) = \T$ is an isomorphism by \eqref{eqn:SdlH-scalars}. Hence, $[\omega] = 0$ in $H^3(G, \T)$.
\end{proof}

When $G$ is countable, the converse to Theorem \ref{thm:JiangSu} is true. There exists an action of $G$ on $\Z \cong \bigotimes_{g \in G}\Z$ defined by permuting the tensor factors. Using Remark \ref{rem:changing-cocycle}, there exists $\omega$-anomalous actions of $G$ on $\Z$ for all $\omega \in Z^3(G, \T)$ with $[\omega] = 0$ in $H^3(G, \T)$. 

We now turn to the proof of Theorem \ref{thm:UHF}. We begin with a preliminary result that is of independent interest.
Given an anomalous-action on $A$, the following lemma will allow us, under certain conditions, to induce anomalous-actions on corners of $A$. We recall that a C$^*$-algebra $A$ is said to have the \emph{cancellation property} if any two projections that agree in $K_0(A)$ are Murray von Neumann equivalent.

\begin{lemma}\label{lem:corners}
Let $G$ be a group and $\omega \in Z^3(G,\T)$. Let $A$ be a unital $\rm{C}^*$-algebra with the cancellation property. Then an $\omega$-anomalous action on $A$ which preserves the $K_0$-class of a projection $p \in A$, induces an $\omega$-anomalous action on $pAp$.
\end{lemma}
\begin{proof}
Suppose $(\theta, u)$ is an $\omega$-anomalous action of $G$ on $A$.
Since $A$ has the cancellation property, there exist partial isometries $v_g\in A$ such that $v_gv_g^*=p$ and $v_g^*v_g=\theta_g(p)$ for each $g\in G$. Define
\begin{align}
\theta_g' &= \Ad(v_g) \circ \theta_g|_{pAp}, \\
u_{g,h}'  &= v_{gh}u_{g,h}\theta_g(v_h^*)v_g^*.
\end{align} 
Then $u_{g,h}' \in U(pAp)$ and we have
\begin{align}
    \Ad(u_{g,h}')\theta'_g\theta'_h &= \Ad(v_{gh}u_{g,h}\theta_g(v_h^*))\theta_g\theta'_h\\
    &= \Ad(v_{gh}u_{g,h})\theta_g\Ad(v_h^*)\theta'_h \nonumber\\
    &= \Ad(v_{gh}u_{g,h})\theta_g\theta_h \nonumber\\
    &= \Ad(v_{gh})\theta_{gh} \nonumber\\
    &= \theta'_{gh}. \nonumber
\end{align}
Computing the 3-cocycle $\omega'(g,h,k)$ associated to $(\theta',u')$ using \eqref{eqn:def-omega}, we find that 
\begin{align}
\omega'(g,h,k)&=u_{g,hk}'\theta_g'(u_{h,k}')u_{g,h}'^*u_{gh,k}'^*\\
&=v_{ghk}u_{g,hk}\theta_g(v_{hk}^*)v_g^* \cdot v_g\theta_g(v_{hk}u_{h,k}\theta_h(v_k^*)v_h^*)v_g^* \nonumber\\
&\quad \quad \cdot v_g\theta_g(v_h)u_{g,h}^*v_{gh}^*\cdot v_{gh}\theta_{gh}(v_k)u_{gh,k}^*v_{ghk}^* \nonumber\\
&= v_{ghk}u_{g,hk}\theta_g(u_{h,k})\theta_g(\theta_h(v_k^*))u_{g,h}^*\theta_{gh}(v_k)u_{gh,k}^*v_{ghk}^* \nonumber\\
&= v_{ghk}u_{g,hk}\theta_g(u_{h,k})u_{g,h}^*\theta_{gh}(v_k^*)\theta_{gh}(v_k)u_{gh,k}^*v_{ghk}^* \nonumber\\
&= v_{ghk}u_{g,hk}\theta_g(u_{h,k})u_{g,h}^*u_{gh,k}^*v_{ghk}^* \nonumber\\
&= v_{ghk}\omega(g,h,k)v_{ghk}^* \nonumber\\
&= \omega(g,h,k)p \nonumber
\end{align}
Therefore, $(\theta',u')$ is an $\omega$-anomalous action of $G$ on $pAp$.
\end{proof}

We are now ready to prove Theorem \ref{thm:UHF}.
\begin{proof}[Proof of Theorem \ref{thm:UHF}]
Let $G$ be a finite group and $\omega\in Z^3(G,\T)$ with $r$ its order in $H^3(G,\T)$. Let $A$ be the UHF algebra $\bigotimes_{k \in \N} \M_{n_k}$ with supernatural number $\mathfrak{n} = \prod_{k \in \N} n_k$. Then $A$ has a unique tracial state $\tau$, which is therefore invariant under all automorphisms. As $A$ is AF, $K_1(A) = 0$.  
    The $K_0$ group of $A$ is isomorphic via $\tau_*$ to the subgroup $Q(\mathfrak{n}) \subseteq \mathbb{R}$ generated by $\{\tfrac{1}{n}: n \in \N, n \, | \, \mathfrak{n}\}$.
    
    Suppose there exists an $\omega$-anomalous action of $G$ on $A$. By Corollary \ref{cor:invariant-trace}, we have $[\Delta_\tau \circ \omega] = 0$ in $H^3(G, \mathbb{R}/Q(\mathfrak{n}))$.  
   The short exact sequence of coefficient groups
    \begin{equation}
        0 \rightarrow \frac{Q(\mathfrak{n})}{\mathbb{Z}} \xrightarrow{\makebox[1cm]{$\iota$}} \T \xrightarrow{\makebox[1cm]{$\Delta_\tau$}} \frac{\mathbb{R}}{Q(\mathfrak{n})} \rightarrow 0,
    \end{equation}
    where $\iota(x) = e^{2\pi i x}$, induces a long exact sequence of cohomology groups; see for example \cite[Section III.6]{BR12}. Therefore, since $\Delta_{\tau*}[\omega] = 0$ in $H^3(G, \mathbb{R}/Q(\mathfrak{n}))$, we have that $[\omega] = \iota_*(\eta)$ for some $\eta \in H^3(G,Q(\mathfrak{n})/\mathbb{Z})$.

    Every element in $Q(\mathfrak{n})/\z$ has order dividing the supernatural number $\mathfrak{n}$. Since $G$ is finite, the same is true for elements of $C^3(G, Q(\mathfrak{n})/\z)$ and so for elements of $H^3(G, Q(\mathfrak{n})/\z)$.
    Therefore, $r$ divides the supernatural number $\mathfrak{n}$.
    
    An inductive argument, based on Lemma \ref{lem:corners}, now shows that in fact $r^\infty$ divides $\mathfrak{n}$.
    Suppose $r^k$ divides $\mathfrak{n}$. Then there exists a projection $p \in A$ of trace $r^{-k}$. As $A$ is a UHF algebra, $A$ has the cancellation property and all automorphisms of $A$ act trivially on $K_0(A)$. Therefore, we may apply Lemma \ref{lem:corners} to obtain an $\omega$-anomalous action on $pAp$. Since $pAp$ is a UHF algebra with supernatural number $r^{-k}\mathfrak{n}$, we can apply the argument above to $pAp$ to get that $r$ divides $r^{-k}\mathfrak{n}$. Hence,  $r^{k+1}$ divides $\mathfrak{n}$.
\end{proof}

\begin{remark}\label{rmk:dividinggroup}
It is a standard result in group cohomology that, for a finite group $G$, every element in $H^3(G,\T)$ has order dividing $|G|$; see for example \cite[Corollary 10.2]{BR12}. This further restricts the possible values of $r$ in Theorem \ref{thm:UHF}. In particular, if $|G|$ is coprime to the supernatural number $\prod_{k\in\N}n_k$, then for any $\omega$-anomalous action we have $[\omega] = 0$ in $H^3(G, T)$.
\end{remark}

We end this section with a general result for C$^*$-algebras with $K_1(A) = 0$, which encompasses both Theorem \ref{thm:JiangSu} and Theorem \ref{thm:UHF}. Its proof requires additional concepts in homological algebra, for which we shall refer the reader to \cite{BR12}.

\begin{proposition}\label{prop:order}
    Let $G$ be a group, $\omega \in Z^3(G,\T)$
    and $A$ be a unital $\rm{C}^*$-algebra with $K_1(A) = 0$.
	Let $(\theta, u)$ be an $\omega$-anomalous action of $G$ on $A$
	and $\tau \in T(A)$ be invariant under $\theta_g$ for all $g \in G$.
	Suppose $[\omega]$ has finite order $r$ in $H^3(G,\T)$.
    Then $\tfrac{1}{r} \in \tau_*(K_0(A))$. If  $[\omega]$ has infinite order, then $\tau_*(K_0(A))$ is dense in $\mathbb{R}$.
\end{proposition}
\begin{proof}
Consider the short exact sequence of abelian groups
\begin{equation}
  0\rightarrow \frac{\tau_*(K_0(A))}{\mathbb{Z}} \xrightarrow{\makebox[1cm]{$\iota$}} \T \xrightarrow{\makebox[1cm]{$\Delta_\tau$}} \frac{\mathbb{R}}{\tau_*(K_0(A))} \rightarrow 0.  
\end{equation}
Applying the universal coefficient theorem for group cohomology (see for example \cite[Proposition 11.9.2]{TD08}), and using the fact that the universal coefficient theorem is natural with respect to the coefficient groups, we obtain the following commuting diagram
\begin{equation*}
\begin{tikzcd}[scale cd=0.8]
\Ext_\z^1(H_2(G,\z),\frac{\mathbb{R}}{\tau_*(K_0(A))})\arrow[r,hook]& H^3(G,\frac{\mathbb{R}}{\tau_*(K_0(A))})\arrow[r,two heads, "\rho_1"] & \Hom_{\z}(H_3(G,\z),\frac{\mathbb{R}}{\tau_*(K_0(A))})\\
 \Ext_\z^1(H_2(G,\z),\T)\arrow[r,hook]\arrow[u]& H^3(G,\T)\arrow[r,two heads, "\rho_2"]\arrow[u,"\Delta_{\tau*}"] & \Hom_{\z}(H_3(G,\z),\T)\arrow[u,"\Delta_{\tau*}"]
\end{tikzcd}
\end{equation*}
where the rows are short exact sequences.

Notice that the Ext terms vanish as both $\T$ and $\mathbb{R}/\tau_*(K_0(A))$ are divisible groups. In particular both $\rho_1$ and $\rho_2$ are isomorphisms and so, the order of $\rho_2([\omega])$ and that of $[\omega]$ is the same. Moreover, $\Delta_{\tau*}([\omega]) = 0$ by Corollary \ref{cor:invariant-trace}. As the diagram commutes, we deduce that  $\rho_2([\omega])$ is a group homomorphism $f:H_3(G, \mathbb{Z}) \rightarrow \T$ that actually takes values in $\tau_*(K_0(A))/\mathbb{Z}$. Since the group operation of $\Hom_\mathbb{Z}(H_3(G, \mathbb{Z}), \T)$ is just pointwise multiplication, the order of $f$ is the same as the exponent of the group $\mathrm{Im}(f) \subseteq \T$.

Suppose $[\omega]$ has finite order $r$. The only subgroup of $\T$ with exponent $r$ is the group of $r$-th roots of unity, so $\mathrm{Im}(f)$ is this subgroup. Since $f$ takes values in $\tau_*(K_0(A))/\mathbb{Z}$, this means that $\frac{1}{r} \in \tau_*(K_0(A))$. Suppose $[\omega]$ has infinite order. Then $\mathrm{Im}(f)$ is a infinite subgroup of $\T$. All such subgroups are dense. It follows that $\tau_*(K_0(A))/\mathbb{Z}$ is dense in $\T$. Therefore, $\tau_*(K_0(A))$ is dense in $\mathbb{R}$.
\end{proof}

Note that Proposition \ref{prop:order} allows us to generalise Theorem \ref{thm:UHF} to the case when $G$ is not necessarily finite but the order of the cocycle $\omega$ is finite.
In Section \ref{sec:nonzero-K1}, we shall show that the conclusion of Proposition~\ref{prop:order} can fail when $K_1(A) \neq 0$.

\section{\texorpdfstring{$\rm{C}^*$}{C*}-Existence}\label{sec:existence}

In this section, we construct examples of anomalous actions of groups on $\rm{C}^*$-algebras. The general strategy is to adapt the von Neumann algebraic results for anomalous actions on $\R$ (see Section \ref{sec:actions-on-R}).

The result is typically an anomalous action on a simple separable nuclear $\rm{C}^*$-algebra $A$ with a unique trace $\tau$, in which case the GNS-closure $\pi_\tau(A)'' \cong \R$. Up to isomorphism, there are many such $\rm{C}^*$-algebras. The fine details of the von Neumann algebraic construction determine which of these C$^*$-algebra arises. However, the results of Section \ref{sec:obstructions} restrict the possible C$^*$-algebras based on the group and the 3-cocyle.   

A good example is the case of anomalous actions of the cyclic group $\mathbb{Z}_n$. When the von Neumann algebraic construction was reviewed in Section \ref{sec:actions-on-R}, we observed that the Connes' automorphisms preserved the UHF algebra $\bigotimes_{k \in \N} \M_n$. Hence, for any $\omega \in Z^3(\mathbb{Z}_n, \T)$, there is an $\omega$-anomalous action of $\mathbb{Z}_n$ on  $\bigotimes_{k \in \N} \M_n$. Note that the size of the group is reflected in the supernatural number of the UHF algebra. 

Conversely, the only UHF algebras for which there exists an $\omega$-anomalous action of  $\mathbb{Z}_n$ for all $\omega \in Z^3(\mathbb{Z}_n, \T)$ are those that absorb $\bigotimes_{k \in \N} \M_n$ tensorially. This follows from Theorem \ref{thm:UHF} and the fact that $H^3(\mathbb{Z}_n, \T)$ is cyclic of order $n$.

We begin this section with a review of Vaughan Jones' construction of anomalous actions on $\R$ using twisted crossed products and its recent adaptation to the $\rm{C}^*$-setting by Corey Jones in \cite{JO20}. We will then prove Theorem \ref{thm:UHF-converse}. We shall end this section with a discussion of \emph{free} anomalous actions and some examples where the C$^*$-algebras have non-trivial $K_1$ group.

\subsection{Anomalous actions on twisted crossed products}\label{sec:twisted-cross-products}

We assume that the reader is familiar with the construction of the crossed product(s) of a unital $\rm{C}^*$-algebra $A$ by a discrete group $G$ with respect to an action $\alpha:G \rightarrow \Aut(A)$. 

The crossed product construction can be generalised by twisting the multiplication in $G$ by a 2-cocyle $c \in Z^2(G, \T)$, i.e. the canonical unitaries $\{v_g: g \in G\}$ in a twisted crossed product satisfy 
\begin{align}
    v_gv_h &= c(g,h)v_{gh}, &(g,h \in G),\\
    v_gav_g^* &= \alpha_g(a), &(g \in G, \, a \in A).
\end{align}
We write $A \rtimes^{\alg}_{\alpha, c} G$ for the \emph{algebraic} twisted crossed product, whose elements can be viewed as (finite) formal sums $\sum_{g \in G} a_gv_g$ with $a_g \in A$. As in the non-twisted case, there are two natural choices of completions: the reduced twisted crossed product  $A \rtimes^{\mathrm{r}}_{\alpha,c} G$ and the maximal twisted crossed product $A \rtimes^{\mathrm{max}}_{\alpha,c} G$; see \cite{BS70} or \cite{PR89} for full details. Note that when $c$ is trivial, we recover the usual crossed products. 

We can now state an existence theorem for anomalous actions, due to Corey Jones (\cite{JO20}), which in turn is based on Vaughan Jones' construction in the von Neumann setting (\cite{JON79}). 
\begin{theorem}
\label{thm:coreys-existence}
    Suppose we have the following data:
    \begin{itemize}
        \item a group $Q$ and $[\omega] \in H^3(Q, \T)$ with a normalised representative $\omega \in Z^3(Q, \T)$;\footnote{An $n$-cochain $\phi:G^n \rightarrow \T$ is said to be \emph{normalised} if $\phi(g_1,\ldots,g_n) = 1$ whenever $g_i=1$ for some $i$. Every cohomology class has a normalised representative; see for example \cite[ Section 6.5]{Wei94}.}
        \item a group G and a surjective homomorphism $\rho:G \twoheadrightarrow Q$ with kernel $K$;
        \item a normalised cochain $c \in C^2(G, \T)$ such that $dc = \rho^*(\omega)$;
        \item a $\rm{C}^*$-algebra $B$ and an action $\pi:G \rightarrow \Aut(B)$.
    \end{itemize}
    Then there exists an $\omega$-anomalous action of $Q$ on the reduced twisted crossed product $B \rtimes^{\mathrm{r}}_{\pi, c} K$, where $c \in Z^2(K,\T)$ is the restriction of $c \in C^2(G,\T)$ to $K$.
\end{theorem}

A detailed proof of Theorem \ref{thm:coreys-existence} can be found in \cite[Theorem 3.1]{JO20}. We provide a brief outline on how the anomalous action is constructed.  

The anomalous action is in fact defined on the algebraic twisted crossed product and shown to extend to the reduced twisted crossed product.\footnote{The action also extends by universality to the full twisted crossed product.} The automorphisms $\theta_q \in \Aut(B \rtimes^{\mathrm{r}}_{\pi, c} K)$ are given by 
\begin{align}
    \theta_q\left(\sum_{k \in K} a_kv_k\right) &= \sum_{k \in K} c(\hat{q}k\hat{q}^{-1},\hat{q})^{-1}c(\hat{q},k)\pi_{\hat{q}}(a_k)v_{\hat{q}k\hat{q}^{-1}},
\end{align}
where $q \mapsto \hat{q}$ is a choice of set theoretic section of $\rho:G \twoheadrightarrow Q$. The unitaries $u_{q,r} \in U(B \rtimes^{\mathrm{r}}_{\pi, c} K)$ are given by 
\begin{align}
    u_{q,r} &= c(\hat{q},\hat{r})^{-1}c(\gamma(q,r),\widehat{qr})v_{\gamma(q,r)}^*,
\end{align}
where  $\gamma:Q \times Q \rightarrow K$ is defined by $\hat{q}\hat{r} = \gamma(q,r)\widehat{qr}$.

\medskip 

In order to access Theorem \ref{thm:coreys-existence}, we will also need the following lemma of a cohomological nature. This lemma also goes back to Vaughan Jones (\cite{JON79}) with the additional observations in the case when $Q$ is finite due to Corey Jones (\cite{JO20}).
\begin{lemma}\label{lem:cohomology-lemma}
    Let $Q$ be a group and $[\omega_0] \in H^3(Q, \T)$. There exist
    \begin{itemize}
        \item a group $G$,
        \item a surjective homomorphism $\rho:G \twoheadrightarrow Q$ with abelian kernel $K$,
        \item  a normalised 2-cochain $c \in C^2(G, \T)$ and a normalised 3 cocycle $\omega \in [\omega_0]$ such that $dc = \rho^*(\omega)$.
    \end{itemize}
    Moreover, when $Q$ is finite, then $K$ can be chosen to be a finite abelian group whose order is a power of $|Q|$, and $c$ can be chosen such that $c|_{K} = 1$. 
\end{lemma}
\begin{proof}
    See  \cite[Lemma 2.3]{JON79} for the general case and \cite[Lemma 3.7]{JO20} for the technical improvements when $Q$ is finite. We observe that, in the proof of \cite[Lemma 3.7]{JO20}, $K$ is defined to be the quotient of $\Hom_{\mathbb{Z}}(\mathbb{Z}Q, \mathbb{Z}_{|Q|})$ by a copy of $ \mathbb{Z}_{|Q|}$. Hence, the order of $K$ is $|Q|^{|Q|-1}$.
\end{proof}

\subsection{Anomalous actions on UHF algebras}

We now prove Theorem \ref{thm:UHF-converse} using the machinery of Section \ref{sec:twisted-cross-products}. For compatibility with the notation of \cite{JON79} and \cite{JO20}, we shall write $Q$ for the starting group and use $G$ for a group extension. 
\begin{theorem}[Theorem \ref{thm:UHF-converse}]\label{thm:UHF-converse-Q}
        Let $Q$ be a finite group and let $\omega_0 \in Z^3(Q, \T)$ be any 3-cocycle. There exists an $\omega_0$-anomalous action of $Q$ on the UHF algebra $\bigotimes_{j \in \N} \M_{|Q|}$. 
\end{theorem}
\begin{proof}
Let $G$, $K$, $\rho:G \twoheadrightarrow Q$, $\omega \in [\omega_0]$ and $c \in C^2(G,\T)$ be as in the conclusion of Lemma \ref{lem:cohomology-lemma} where $|K|$ is a power of $|Q|$ and $c|_{K} = 1$. Note that $|G| = |K||Q|$ is finite.  
Let $\lambda_G:G \rightarrow U(\B(\ell^2(G)))$ denote the left regular representation of $G$. Write $\Ad(\lambda_G)$ for the induced action on $\B(\ell^2(G))$ given by $\Ad(\lambda_G)_g(T) = \lambda_G(g)T\lambda_G(g)^*$ for all $T \in \B(\ell^2(G))$ and $g \in G$. Let $B=\bigotimes_{j\in\N} \B(\ell^2(G))$ and let $\pi=\Ad(\lambda_G)^{\otimes \infty}$.

Applying Theorem \ref{thm:coreys-existence}, we obtain an $\omega$-anomalous action of $Q$ on $B \rtimes^{\mathrm{r}}_{\pi,c} K$. By Remark \ref{rem:changing-cocycle}, we also obtain an $\omega_0$-anomalous action of $Q$ on $B \rtimes^{\mathrm{r}}_{\pi,c} K$.
The remainder of this proof consists of demonstrating that this twisted crossed product is in fact isomorphic to the UHF algebra with supernatural number $|Q|^\infty$.

First, we observe that, since  $c|_K = 1$, the twisted crossed product $B \rtimes^{\mathrm{r}}_{\pi,c} K$ is in fact not twisted in this case. Moreover, as $K$ is finite, there is no distinction to be made between the algebraic, the reduced and the full crossed products. Therefore, we shall simplify our notation and write $B \rtimes_\pi K$ instead of $B \rtimes^{\mathrm{r}}_{\pi,c} K$. 

Next, we consider the restriction of $\lambda_G$ to $K$. By decomposing $G$ into right $K$-cosets, we see that $\lambda_G|_K$ is equivalent to $|G/K| = |Q|$ copies of the left regular representation $\lambda_K:K \rightarrow \B(\ell^2(K))$. Hence, $\lambda_G|_K$ is equivalent to $\lambda_K \otimes 1_{\ell^2(Q)}$ where $1_{\ell^2(Q)}$ denotes the trivial representation of $K$ on the Hilbert space $\ell^2(Q)$.

It follows that we have an equivariant isomorphism of $\rm{C}^*$-algebras
\begin{align}
    \bigotimes_{j\in\N} \B(\ell^2(G)) \cong \bigotimes_{j\in\N} (\B(\ell^2(K)) \otimes \B(\ell^2(Q)),
\end{align}
where $K$ acts by $\pi$ on the left hand side and by $\sigma := \Ad(\lambda_K \otimes 1_{\ell^2(Q)})^{\otimes \infty}$ on the right hand side. 

Taking crossed products, we have
\begin{align}
    B \rtimes_{\pi} K &\cong \left(\bigotimes_{j\in\N} \B(\ell^2(K)) \otimes \B(\ell^2(Q))\right) \rtimes_{\sigma} K \label{eqn:reduction-step1}\\
    &\cong \left(\bigotimes_{j\in\N} \B(\ell^2(K)) \rtimes_{\Ad(\lambda_K)^{\otimes \infty}} K\right) \otimes \B(\ell^2(Q))^{\otimes \infty}.\nonumber 
\end{align}
Since $\B(\ell^2(Q))^{\otimes \infty}$ is a UHF algebra with supernatural number $|Q|^\infty$, it suffices to show that $\bigotimes_{k\in\N} \B(\ell^2(K)) \rtimes_{\Ad(\lambda_K)^{\otimes \infty}} K$ is isomorphic to a UHF algebra with supernatural number $|K|^\infty$. 

The infinite tensor product of the left regular representation has the Rokhlin property (see for example \cite[Example 3.2]{IZU04}). Hence, by \cite[Corollary 4.5]{OSTE14} (or by \cite[Theorem 4.4]{KI77} for cyclic groups) the fixed point algebra $(\bigotimes_{k\in\N} \B(\ell^2(K))^K$ is isomorphic to the UHF algebra $\bigotimes_{k\in\N} \B(\ell^2(K))$. Moreover, the fixed point algebra $(\bigotimes_{k\in\N} \B(\ell^2(K)))^K$ is isomorphic to the crossed product $\bigotimes_{k\in\N} \B(\ell^2(K)) \rtimes_{\Ad(\lambda_K)^{\otimes \infty}} K$ by \cite[Theorem 2.16]{SZBA17}. This completes the proof.
\end{proof}
As a corollary of Theorem \ref{thm:UHF-converse}, we obtain anomalous actions on UHF-stable algebras.
\begin{corollary}\label{cor:UHF-stable}
    Let $G$ be a finite group and let $\omega \in Z^3(G, \T)$ be any 3-cocycle. Let $A$ be a unital $\rm{C}^*$-algebra that tensorially absorbs the UHF algebra $\bigotimes_{j \in \N} \M_{|G|}$. Then there exists an $\omega$-anomalous action of $G$ on $A$. 
\end{corollary}
\begin{proof}
    By Theorem \ref{thm:UHF-converse}, there exists an $\omega$-anomalous action $(\theta, u)$ on $\bigotimes_{j \in \N} \M_{|G|}$.
    Set $\theta_g' = \id_A \otimes \theta_g$ and $u_{g,h}' = 1_A \otimes u_{g,h}$.
    Then $(\theta', u')$ is an $\omega$-anomalous action of $G$ on $A \otimes \bigotimes_{j \in \N} \M_{|G|} \cong A$.
\end{proof}

Examples of $\rm{C}^*$-algebras that absorb the UHF algebra $\bigotimes_{j \in \N} \M_{|G|}$ tensorially include UHF algebras whose supernatural number is divisible by $|G|^\infty$ as well as the Cuntz algebras $\mathcal{O}_{|G|}$ and $\mathcal{O}_2$. 

In the non-unital setting, an interesting example is the Jacelon–-Razak algebra $\mathcal{W}$ (see \cite{Jac13}). The proof of Corollary \ref{cor:UHF-stable} is easily adapted to the non-unital setting with unitaries $u_{g,h}'$ now living in the multiplier algebra $M(\mathcal{W})$. In general, the unitaries can not be taken cannot live in the minimal unitisation $\mathcal{W}^\sim$, else we could apply Proposition \ref{HomtoAb} with $q:U(A^\sim)\rightarrow\T$ the restriction of the quotient map $A^\sim \twoheadrightarrow A^\sim / A \cong \C$ to obtain that $[\omega] = 0$.

Theorem \ref{thm:UHF-converse} provides us with a partial converse to Theorem \ref{thm:UHF}. By Theorem \ref{thm:UHF}, a finite group $G$ has an $\omega$-anomalous action on a UHF algebra with supernatural number $\mathfrak{n}$ only if the order $r_\omega$ of $[\omega] \in H^3(G, \T)$ satisfies $r_\omega^\infty | \mathfrak{n}$. Hence, all $\omega$-anomalous actions exist only if the exponent $e$ of $H^3(G, \T)$ satisfies $e^\infty | \mathfrak{n}$. Since $|G|$ annihilates $H^3(G, \T)$, the exponent $e$ of is a factor of $|G|$ in general. Note that when $G$ is a cyclic group $|G|^\infty = e^\infty$.


\subsection{Free anomalous actions}

Following the terminology of \cite{OC06} for group actions on von Neumann algebras, we will call an $\omega$-anomalous action $(\theta,u)$ \emph{free} if $\theta_g \not\in \Inn(G)$ for $g \neq 1_G$. Actions with this property have also been called $\emph{outer}$ actions in the literature (e.g.\ \cite{JON80}).

A non-free $\omega$-anomalous action $(\theta,u)$ of $G$ will descend to an $\omega'$-anomalous action of the quotient of $G/K$ where $K = \theta^{-1}(\Inn(A))$. Moreover, $[\omega]$ will be the pullback of $[\omega']$ under the quotient map. This observation can be used to prove that freeness is automatic is certain situations, for example if $[\omega]$ has order $|G|$.

The following tensor product trick can be used to obtain free anomalous actions on $\Z$-stable C$^*$-algebras.
\begin{proposition}
    Let $(\theta, u)$ be an $\omega$-anomalous action of a countable discrete group $G$ on a C$^*$-algebra $A$. Let $\alpha:G \curvearrowright \bigotimes_{g \in G} \Z$ be the action given by permuting the tensor factors. Set $\theta_g' = \theta_g \otimes \alpha_g$ and $u_{g,h}' = u_{g,h} \otimes 1_\Z$. Then $(\theta', u')$ is a free $\omega$-anomalous action of $G$ on $A \otimes (\bigotimes_{g \in G} \Z)$.
\end{proposition}
\begin{proof}
Since $\alpha$ is an action, it follows immediately from Definition \ref{def:anomalous-action} that $(\theta', u')$ is an $\omega$-anomalous action. It remains to show that $(\theta', u')$ is free.

Let $(z_n)_{n=1}^\infty$ be a central sequence in $\Z$ where each $z_n$ is a self-adjoint element with spectrum $[-1,1]$. Set $x_n = 1_A \otimes z_n \otimes (\bigotimes_{g \neq 1_G} 1_\Z)$. Then 
\begin{align}
    \|\theta_g'(x_n) - x_n\| &= \|z_n \otimes 1_\Z - 1_\Z \otimes z_n\|_{\Z \otimes \Z} \\
    &= \sup_{s,t \in [-1,1]}|s - t| \nonumber\\
    &= 2 \nonumber
\end{align}
for all $n \in \N$ and $g \neq 1_G$. However, $\lim_{n\to \infty} \|\phi(x_n) - x_n\| = 0$ for all inner automorphisms $\phi$, since $z_n$ is a central sequence in $\Z$.
\end{proof}

\subsection{Examples with non-trivial \texorpdfstring{$K_1$}{K1} groups} \label{sec:nonzero-K1}

In this section, we construct anomalous actions of finite cyclic groups on the irrational rotation algebras and the Bunce--Deddens algebras. 

These examples could be constructed with the general machinery of Section \ref{sec:twisted-cross-products}. However, we will take a more concrete approach. Indeed, since we are working with a cyclic group $\mathbb{Z}_n$ it suffices to construct an automorphism $\phi \in \Aut(A)$ and a unitary $u \in U(A)$ such that $\phi^n = \Ad(u)$ and $\phi(u) = \gamma u$ for some $n$-th root of unit $\gamma \in \C$. The same arguments made in Section \ref{sec:actions-on-R}, where we discussed Connes' automorphisms of $\R$, apply in this case. 

\begin{proposition}[{cf. \cite[Corollary 3.6]{JO20}}]\label{prop:A-theta}
    Let $\theta \in \mathbb{R}\setminus\mathbb{Q}$. Let $A_\theta$ be the irrational rotation algebra. For any $n \in \mathbb{N}$ and any $n$-th root of unity $\gamma \in \C$, there exists $\phi \in \Aut(A_\theta)$ and a unitary $u \in U(A)$ such that
\begin{align}
      \phi^n &= \Ad(u),\label{eqn:A-theta-1}\\
      \phi(u) &= \gamma u. \label{eqn:A-theta-2}   
\end{align}
\end{proposition}
\begin{proof}
    We view $A_\theta$ as the universal C$^*$-algebra generated by two unitaries $u, v \in A_\theta$ such that $uv = e^{2\pi i \theta}vu$.
    
    By the universal property of $A_\theta$, we may define $\phi \in \Aut(A_\theta)$ by setting $\phi(u) = \gamma u$ and $\phi(v) = e^{2 \pi i \theta/n}v$. By construction \eqref{eqn:A-theta-2} holds. We compute that $\phi^n(u) = u$ and $\phi^n(v) = e^{2\pi i \theta}v = uvu^*$. Since $u$ and $v$ generate $A_\theta$, we have \eqref{eqn:A-theta-1}. 
\end{proof}

The irrational rotation algebra $A_\theta$ has a unique trace $\tau$. Moreover, $K_0(A_\theta) \cong \mathbb{Z} + \mathbb{Z}\theta \subseteq \mathbb{R}$ with the isomorphism induced by the trace $\tau$ (see \cite[Example VIII.5.1]{Dav96}). It follows from Proposition \ref{prop:A-theta} that there are $\omega$-anomalous actions of $\mathbb{Z}_n$ on $A_\theta$ for all $n \in \N$ and all $[\omega] \in H^3(\mathbb{Z}_n, \T)$. This does not contradict Proposition \ref{prop:order} since $K_1(A_\theta) \cong \mathbb{Z} \oplus \mathbb{Z}$.

We now turn to the Bunce--Deddens algebras (see \cite[Section V.3]{Dav96}). These algebras arise as the crossed product of a Cantor space by an odometer action. Like the UHF algebras, Bunce--Deddens algebras are classified up to isomorphism by a supernatural number $\mathfrak{n} = \prod_{k \in \N} n_k$. In the odometer construction $n_k$ is the number of values on the $k$-th dial of the odometer (see \cite[Theorem VIII.4.1]{Dav96}). The Bunce--Deddens algebra $B_\mathfrak{n}$ has a unique trace $\tau$, $K_0(B_\mathfrak{n}) = Q(\mathfrak{n})$ and $K_1(B_\mathfrak{n}) = \mathbb{Z}$.

\begin{proposition}\label{prop:BunceDeddens}
    Let $\mathfrak{n} = \prod_{k \in \N} n_k$ be a supernatural number. Let $B_\mathfrak{n}$ be the corresponding Bunce--Deddens algebra. For any $m \in \mathbb{N}$, which is coprime to $\mathfrak{n}$ and any $m$-th root of unity $\gamma \in \C$, there exists $\phi \in \Aut(B_\mathfrak{n})$ and a unitary $u \in U(A)$ such that
\begin{align}
      \phi^m &= \Ad(u),\label{eqn:BD-1}\\
      \phi(u) &= \gamma u. \label{eqn:BD-2}   
\end{align}
\end{proposition}
\begin{proof}
    Let $X=\prod_{k \in \N} X_k$ where $X_k$ is a discrete topological space with $n_k$ points. Let $\alpha \in \operatorname{Homeo}(X)$ be the odometer map and $\alpha^* \in \Aut(C(X))$ be the induced automorphism. Then $B_\mathfrak{n} \cong C(X) \rtimes_{\alpha^*} \mathbb{Z}$.
    
    The key to the construction is the observation that $\alpha$ has an $m$-th root whenever $m$ is coprime to $\mathfrak{n}$. In the case where $\mathfrak{n}=p^\infty$, this is just the observation that $m$ is invertible in the ring of $p$-adic integers. In general, we work in the topological ring $R$ that arises as the inverse limit of the system
    \begin{equation}
        \cdots \rightarrow \frac{\mathbb{Z}}{n_4n_3n_2n_1\mathbb{Z}} \rightarrow \frac{\mathbb{Z}}{n_3n_2n_1\mathbb{Z}} \rightarrow \frac{\mathbb{Z}}{n_2n_1\mathbb{Z}}\rightarrow \frac{\mathbb{Z}}{n_1\mathbb{Z}}.
    \end{equation}
    We identify $X$ with the underlying topological space of $R$ and $\alpha$ with addition by $1_R$. Since the image of $m\cdot1_R$ is invertible at each stage of the system, it is invertible in $R$. Let $\beta \in \operatorname{Homeo}(X)$ be given by addition by $(m \cdot 1_R)^{-1}$.
    
    Let $u \in B_\mathfrak{n}$ be the canonical unitary of the crossed product. By the universal property of the crossed product, we may define $\phi \in \Aut(B)$ by $\phi(u) = \gamma u$ and $\phi(f) = f \circ \beta$ for all $f \in C(X)$. Note that $\alpha$ and $\beta$ commute since $\alpha = \beta^m$. By construction \eqref{eqn:BD-2} holds, and \eqref{eqn:BD-1} follows since $\phi^m(u) = u = \Ad(u)(u)$ and $\phi^m(f) = \Ad(u)(f)$ for all $f \in C(X)$.
\end{proof}

\section{\texorpdfstring{$\rm{C}^*$}{C*}-Tensor categories}\label{sec:UTC}

In this section, we show how to translate our results on anomalous actions to the language of $\rm{C}^*$-tensor categories and $\rm{C}^*$-tensor functors. We begin with an introduction to the basic terminology and the key examples.

\subsection{\texorpdfstring{$\rm{C}^*$}{C*}-Tensor categories}

We assume the reader is familiar with the basic language of category theory (see for example \cite{MAC13}). All categories that we shall consider will be \emph{$\C$-linear categories}, meaning that the space of morphisms $\Hom(X,Y)$ between any two objects of the category is endowed with a $\C$-vector space structure and composition of morphisms is bilinear.

\begin{definition}{(\cite{GLR85}; see also \cite[Section 2]{JP16})}
    A \emph{$\rm{C}^*$-category} is a $\C$-linear category $\mathcal{C}$ equipped with a conjugate linear maps $^*\!:\Hom(X,Y) \rightarrow \Hom(Y,X)$ for every $X,Y \in \mathcal{C}$ such that
    \begin{enumerate}
        \item $\phi^{**} = \phi$ for all $\phi \in \Hom(X,Y);$ 
        \item $(\phi \circ \psi)^* = \psi^* \circ \phi^*$ for all $\psi \in \Hom(X,Y), \phi \in \Hom(Y,Z)$;
        \item The function $\|\cdot\|:\Hom(X,Y) \rightarrow [0,\infty]$ given by
        \begin{equation*}
            \|\phi\|^2 = \sup \{\lambda > 0: \phi^* \circ \phi - \lambda\id_X \mbox{ is not invertible}\}
        \end{equation*}
        is a complete norm on $\Hom(X,Y)$;
    
    \item $\|\phi \circ \psi\| \leq \|\phi\|\|\psi\|$ for all $\psi \in \Hom(X,Y), \phi \in \Hom(Y,Z)$;
    \item $\|\phi^* \circ \phi\| = \|\phi\|^2$ for all $\phi \in \Hom(X,Y)$;
    \item For all $\phi \in \Hom(X,Y)$, $\phi^* \circ \phi$ is a positive element of the $\rm{C}^*$-algebra $\Hom(X,X)$.
    \end{enumerate}
A \emph{$\rm{C}^*$-functor} $F:\mathcal{C} \rightarrow \mathcal{D}$ between $\rm{C}^*$-categories is required to be $\C$-linear on morphisms and $^*$-preserving. A natural isomorphism $\nu:F \rightarrow G$ is said to be unitary if $\nu_x \in \Hom(F(X), G(X))$ satisfies $\nu_X^* \circ \nu_X = 1_{F(X)}$ and $\nu_X \circ \nu_X^* = 1_{G(X)}$ for all $X \in \mathcal{C}$.
\end{definition}

A simple example of a $\rm{C}^*$-category is $\Hilb_\C$, whose objects are finite-dimensional Hilbert spaces over $\C$ and morphism are linear maps. The map $\phi \mapsto \phi^*$ is the Hilbert space adjoint. More generally, for any $\rm{C}^*$-algebra $A$, the (right) Hilbert $A$-modules and adjointable maps form a $\rm{C}^*$-category (see for example \cite{Lance95}).  

We now consider tensor product structures on a $\rm{C}^*$ category.
\begin{definition}{(see for example \cite{EGNO, JP16})}
    A \emph{$\rm{C}^*$-tensor category} is a $\rm{C}^*$-category $\mathcal{C}$ together with a $\C$-linear bifunctor $-\otimes-:\mathcal{C}\times\mathcal{C} \rightarrow \mathcal{C}$, a distinguished object $1_\mathcal{C} \in \C$ and unitary natural isomorphisms   
        \begin{align}
            \alpha_{X,Y,Z}&: (X \otimes Y) \otimes Z \rightarrow X \otimes (Y \otimes Z),\\
            \lambda_X&: (1_\mathcal{C} \otimes X) \rightarrow X,\nonumber\\
            \rho_X&: ( X \otimes 1_\mathcal{C}) \rightarrow X,\nonumber
        \end{align} 
such that $(\phi \otimes \psi)^* = (\phi^* \otimes \psi^*)$ and the following diagrams commute for any $X,Y,Z,W\in \mathcal{C}$
    
\begin{equation}
\begin{tikzcd}[column sep=0.3em]
 {}
&((W\otimes X)\otimes Y)\otimes Z \arrow{dl}{\alpha_{W,X,Y}\otimes \operatorname{id_Z}} \arrow{drr}{\alpha_{W\otimes X,Y,Z}}
& {}  \\
(W\otimes(X\otimes Y))\otimes Z \arrow{d}{\alpha_{W,X\otimes Y,Z}}
& & & (W\otimes X)\otimes (Y\otimes Z) \arrow{d}{\alpha_{W,X,Y\otimes Z}} \\
W\otimes ((X\otimes Y)\otimes Z) \arrow{rrr}{\operatorname{id_W}\otimes \alpha_{X,Y,Z}}
& & & W\otimes (X\otimes(Y\otimes Z)),
\end{tikzcd}
\end{equation}
\vspace{-1em}

\begin{equation}
\begin{tikzcd}
(X\otimes 1_{\mathcal{C}})\otimes\arrow{rr}{\alpha_{X,1,Y}}\arrow{dr}[swap]{\rho_X\otimes \id_Y} Y& {} & X\otimes (1_{\mathcal{C}}\otimes Y)\arrow{dl}{\id_X\otimes\lambda_Y}\\
{}&X\otimes Y.&{}
\end{tikzcd}
\end{equation}
\end{definition}

The $\rm{C}^*$-category $\Hilb_\C$ can be endowed with the additional structure of a $\rm{C}^*$-tensor category by taking $-\otimes-$ to be the Hilbert space tensor product, $1_{\Hilb_\C}$ to be the 1-dimensional Hilbert space $\C$, and taking
\begin{align}
    \alpha_{X,Y,Z}&:(x \otimes y) \otimes z \mapsto x \otimes (y \otimes z),\\
    \lambda_X&:1_{\C} \otimes x \mapsto x,\nonumber\\
    \rho_X&:x \otimes 1_{\C}, \mapsto x\nonumber
\end{align}
for $X,Y,Z$ Hilbert spaces and $x,y,z$ in $X,Y,Z$ respectively. We now state the additional examples of $\rm{C}^*$-tensor categories that we will be using in this paper.
\begin{example}
In the following examples, $G$ is a discrete group,  $\omega \in Z^3(G, \T)$ is a 3-cocyle, and $A$ is a unital $\rm{C}^*$-algebra.
\begin{enumerate}
    \item $\Hilb_\C(G)$: The objects are finite-dimensional, $G$-graded Hilbert spaces, i.e.\ finite-dimensional Hilbert spaces $X$ endowed with a decomposition $X = \bigoplus_{g \in G} X_g$. The morphisms are linear maps that preserve the $G$-grading. The tensor product is the usual Hilbert space tensor product with the $G$-grading defined by $(X \otimes Y)_g = \bigoplus_{h \in G} X_{h} \otimes Y_{h^{-1}g}$. The remaining structure is the same as for $\Hilb_\C$.  
    
    \item $\Hilb_\C(G, \omega)$: Defined exactly the same as $\Hilb_\C(G)$ except that the associators are now given by 
    \begin{align*}
        \alpha_{X,Y,Z}: (x \otimes y) \otimes z \mapsto \omega(g,h,k) \, x \otimes (y \otimes z)
    \end{align*}
    for $x \in X_g$, $y \in Y_h$, $z \in Z_k$.
    
    \item $\Bim(A)$: The objects are (right) Hilbert $A$-modules $X$ endowed with a unital $^*$-homomorphism $A \rightarrow \mathcal{L}(X)$. The morphisms are the adjointable $A$-$A$ bilinear maps. The tensor product is the internal tensor product. (See for example \cite[Chapter 1]{EKQR06}.)
\end{enumerate}
\end{example}

Finally, we recall the notion of a $\rm{C}^*$-tensor functor.
\begin{definition}
    A $\rm{C}^*$-tensor functor $(F,J):\mathcal{C} \rightarrow \mathcal{D}$ is a $\rm{C}^*$-functor $F:\mathcal{C} \rightarrow \mathcal{D}$ such that $F(1_\mathcal{C}) \cong 1_\mathcal{D}$ together with unitary natural isomorphisms $J_{X,Y}:F(X) \otimes F(Y) \rightarrow F(X \otimes Y)$ such that the following diagram commutes
\begin{equation}\label{eqn:monodial-axiom}
\begin{tikzcd}[column sep=5em]
(F(X)\otimes F(Y))\otimes F(Z)\arrow[r,"\alpha_{F(X),F(Y),F(Z)}"]\arrow[d,"J_{X,Y}\otimes\id_{F(Z)}"]  & F(X)\otimes (F(Y) \otimes F(Z)) \arrow[d,"\id_{F(X)}\otimes J_{Y,Z}"] \\
F(X \otimes Y)\otimes F(Z) \arrow[d,"J_{X \otimes Y,Z}"] & F(X)\otimes F(Y \otimes Z) \arrow[d,"J_{X,Y \otimes Z}"]\\
F((X \otimes Y) \otimes Z )\arrow[r,"F(\alpha_{X,Y,Z})"] & F(X \otimes (Y \otimes Z)).
\end{tikzcd}
\end{equation}
\end{definition}

We are particularly interested in \emph{fully faithful} C$^*$-tensor functors $F:\mathcal{C}\rightarrow\mathcal{D}$, i.e.\ functors for which the the induced map $\Hom(X,Y) \rightarrow \Hom(F(X),F(Y))$ is an isomorphism for all $X, Y \in \mathcal{C}$.

\subsection{Anomalous actions and \texorpdfstring{$\rm{C}^*$}{C*}-tensor functors}

We now consider the relationship between anomalous actions of a group $G$ on a $\rm{C}^*$-algebra $A$ and $\rm{C}^*$-tensor functors $\Hilb_\C(G,\omega) \rightarrow \Bim(A)$. We begin by recalling an important class of bimodules.

\begin{example}{(cf. \cite[Section 3]{BRGRRI77})}\label{ex:bimodules}
Let $A$ be a unital $\rm{C}^*$-algebra and $\theta \in \Aut(A)$. Let $A_\theta$ be the (right) Hilbert $A$-$A$-bimodule with underlying vector space $A$, where the $A$-actions and (right) $A$-inner product are given by
\begin{align}
    a \cdot x \cdot b &= ax\theta(b),\\
    \langle x,y \rangle &= \theta^{-1}(x^*y).
\end{align}

Suppose $\theta, \phi \in \Aut(A)$. Any morphism $f \in \Hom(A_\theta, A_\phi)$ must be given by right multiplication by $f(1)$, and $f(1)$ must intertwine the right $A$-actions. It follows that $A_\theta$ and $A_\phi$ are unitary isomorphic if and only if there is a unitary $u \in A$ with $\theta = \Ad(u)\phi$.
Moreover, $\Hom(A_\theta, A_\theta) = Z(A)$, so $A_\theta$ is a simple bimodule when $A$ has a trivial centre. We also have $A_\theta \otimes A_\phi \cong A_{\theta\circ\phi}$ via the unitary isomorphism $J(x \otimes y) = x\theta(y)$.
\end{example}

We now record the construction of a tensor functor from an anomalous action. In the following proposition, we write $\bar{\omega}$ for the complex conjugate of $\omega \in Z^3(G, \T)$, and we write
$\C_g$ for the  Hilbert space $\C$ viewed as as $G$-graded Hilbert space that is homogeneous of degree $g$. 
\begin{proposition}\label{prop:tensor-functor}
    Let $G$ be a group, $\omega \in Z^3(G, \T)$, and $A$ be a unital $\rm{C}^*$-algebra. Given an $\omega$-anomalous action $(\theta,u)$ of $G$, there exists a $\rm{C}^*$-tensor functor $(F, J):\Hilb_\C(G, \bar{\omega}) \rightarrow \Bim(A)$ such that
    \begin{align}
        F(\C_g) &= A_{\theta_g},\\
        J_{{\C_g},{\C_h}}(x \otimes y) &= x\theta_g(y)u_{g,h}^*,
    \end{align}
    which is fully faithful whenever $Z(A) = \C$ and $(\theta,u)$ is free. 
    
    Conversely, if $(F, J):\Hilb_\C(G, \bar{\omega}) \rightarrow \Bim(A)$ is such that for each $g\in G$, $F(\C_g) \cong A_{\theta_g}$ for some $\theta_g \in \Aut(G)$. Then there exists a $\omega$-anomalous action of $G$ on $A$. 
\end{proposition}
\begin{proof}
    We define $F$ on a general object $X = \bigoplus_{g \in G} X_g$ of $\Hilb_\C(G, \bar{\omega})$ by $F(X) = \bigoplus_{g \in G} X_g \otimes_{\C} A_{\theta}$. On a general morphisms $f = \bigoplus_{g \in G} f_g \in \Hom(X,Y)$, we define $F(f) = \bigoplus_{g \in G} f_g \otimes_\C \id_{A_{\theta_g}}$.
    
    Since $(\theta,u)$ is an $\omega$-anomalous action, after taking adjoints, we have
    \begin{align}
        u_{g,h}^*\theta_{gh}(a) &= \theta_g(\theta_h(a))u_{g,h}^*,\label{eqn:anomalous-adjoint-1}\\
        \bar{\omega}(g,h,k)u_{g,h}^*u_{gh,k}^* &= \theta_g(u_{h,k}^*)u_{g,hk}^*\label{eqn:anomalous-adjoint-2},
    \end{align}
    for $a \in A$ and $g,h,k \in G$. 
    It follows that $J_{{\C_g},{\C_h}}(x \otimes y) = x\theta_g(y)u_{g,h}^*$ defines a unitary isomorphism of bimodules $A_{\theta_g} \otimes A_{\theta_h} \cong A_{\theta_{gh}}$, and the monoidal structure axiom \eqref{eqn:monodial-axiom} holds when $X = \C_g$, $Y=\C_h, Z=\C_k$. 
    The definition of the tensorator $J_{X,Y}$ for general objects $X = \bigoplus_{g \in G} X_g$ and $Y = \bigoplus_{h \in G} Y_h$ is uniquely determined by naturality.  It has the form $J_{X,Y} = \bigoplus_{g,h \in G} \id_{X_g} \otimes_\C \id_{Y_h} \otimes_\C \, J_{\C_g, \C_h}$. The monoidal structure axiom remains valid by naturality.
    
    Suppose $Z(A) = \C$ and $(\theta,u)$ is a free anomalous action.  Then $\Hom(A_{\theta_g},A_{\theta_h}) = 0$ whenever $g \neq h$ and $\Hom(A_{\theta_g},A_{\theta_g}) \cong \C$; see Example \ref{ex:bimodules}. It follows that $F$ is fully faithful.
    
    For the converse, suppose $(F, J):\Hilb_\C(G, \bar{\omega}) \rightarrow \Bim(A)$ is a $\rm{C}^*$-tensor functor and for all $g\in G$, $F(\C_g)$ is unitary isomorphic to $A_{\theta_g}$ for some $\theta_g \in \Aut(G)$. Let $g, h \in G$. Then
    \begin{align}
        A_{\theta_{gh}} \cong  F(\C_g \otimes \C_h) \cong F(\C_g) \otimes F(\C_h) \cong A_{\theta_{g}\theta_h}.
    \end{align} 
    This unitary isomorphism of bimodules must be implemented by right multiplication by a unitary $u_{g,h}^* \in U(A)$ satisfying \eqref{eqn:anomalous-adjoint-1}. Moreover, \eqref{eqn:anomalous-adjoint-2} holds since $F$ and $J$ satisfy the monoidal structure axiom \eqref{eqn:monodial-axiom}. Hence, $(\theta,u)$ is an $\omega$-anomalous action of $G$ on $A$.
\end{proof}

\begin{remark}\label{rem:why-omega-bar}
    The fact that an $\omega$-anomalous action corresponds to a functor from $\Hilb_\C(G, \bar{\omega})$ instead of $\Hilb_\C(G, \omega)$ results from a discrepancy of conventions between \cite{BRGRRI77} and \cite{JO20}. One way to resolve this issue is to reverse the order of the tensor product in $\Bim(A)$ and work with bimodules of the form ${}_{\theta_g} A_{\id}$. This is the approach used in \cite{BMZ13}.
\end{remark}

In order to access the converse of Proposition \ref{prop:tensor-functor}, we need a lemma for determining when an $A$-$A$ bimodule is unitary isomorphic to one of the form $A_\theta$ for some $\theta \in \Aut(A)$. To this end we recall that a object $X$ in a $\rm{C}^*$-tensor category $\mathcal{C}$ is \emph{invertible} if there exists an object $Y \in \mathcal{C}$ with $X \otimes Y \cong 1_\mathcal{C} \cong Y \otimes X$. We shall say that $X$ has \emph{finite order} if $X^{\otimes N} \cong 1_\mathcal{C}$ for some $N \in \N$.

\begin{lemma}\label{lem:invertibles}
Let $A$ be a unital, simple $C^*$-algebra with the cancellation property. 
If $\Aut(K_0(A),K_0(A)_+)$ is trivial, then every invertible element $X \in \Bim(A)$ is unitary isomorphic to a bimodule of the form $A_{\theta}$ for some $\theta\in\Aut(A)$. 

Similarly, if $\Aut(K_0(A),K_0(A)_+)$ has no non-trivial elements of finite order, then every invertible element $X \in \Bim(A)$ of finite order is unitary isomorphic to a bimodule of the form $A_{\theta}$ for some $\theta\in\Aut(A)$.
\end{lemma}
\begin{proof}
Let $X \in \Bim(A)$ be an invertible bimodule. Then $X$ is a self-Morita equivalence of $A$ by \cite[Lemma 2.4]{EKQR06}. Let $H = {}_A\ell^2(A)_{A \otimes \K}$ be the standard Morita equivalence between $A$ and its stabilisation $A \otimes \K$. Then $Y = \bar{H} \otimes X \otimes H$ is a self-Morita equivalence of $A \otimes \K$. 
As $A\otimes\mathbb{K}$ is stable and $\sigma$-unital, $Y\cong (A\otimes\mathbb{K})_\theta$ for some $\theta\in\Aut(A\otimes\mathbb{K})$ by \cite[Corollary 3.5]{BRGRRI77}. Let $\theta_* \in \Aut(K_0(A), K_0(A)_+)$ be the induced automorphism. 

Suppose $\Aut(K_0(A), K_0(A)_+)$ is trivial. Then $\theta_* = \id_{K_0(A)}$. Therefore, $[\theta(1_A\otimes e_{11})]_{K_0(A)} = [1_A\otimes e_{11}]_{K_0(A)}$. Since $A$ has the cancellation property, there exists a partial isometry $v\in A\otimes\mathbb{K}$ with $v^*v=1\otimes e_{11}$ and $vv^*=\theta(1\otimes e_{11})$. The series
\begin{equation}
    u=\sum_{i=1}^{\infty} \theta(1\otimes e_{i1})v(1 \otimes e_{1i})
\end{equation}
converges in the strict topology on $M(A \otimes \K)$ and defines a unitary $u\in U(M(A\otimes\K))$ such that $\theta(1 \otimes e_{ij})=\Ad(u)(1 \otimes e_{ij})$ for all $i,j \in \N$. It follows that $\Ad(u^*)\theta$ fixes $1 \otimes \K$, and so is of the form $\theta' \otimes \id_\K$ for some $\theta' \in \Aut(A)$. Hence, $Y \cong (A\otimes\mathbb{K})_{\theta' \otimes \id_\K}$, and we have $X \cong H \otimes Y \otimes \bar{H} \cong A_{\theta'}$, as required. 

If $X^{\otimes N} \cong 1_{\Bim(A)}$, then in the above argument $\theta^N$ is an inner automorphism, so $\theta_*^N = \id_{K_0(A)}$. Hence, in this case, it suffice to know that $\Aut(K_0(A), K_0(A)_+)$ has no non-trivial elements of finite order.
\end{proof}

We are now ready to prove Corollary \ref{cor:UTC-JiangSu} and Corollary \ref{cor:UTC-UHF}.
\begin{proof}[Proof of Corollary \ref{cor:UTC-JiangSu}]
Suppose there exists a $\rm{C}^*$-tensor functor $(F,J):\Hilb_\C(G, \omega) \rightarrow \Bim(\Z)$. As $K_0(\Z) = \mathbb{Z}$ with its usual order, there are no non-trivial elements of $\Aut(K_0(\Z),K_0(\Z)_+)$. Moreover, $\Z$ is simple, unital and has the cancellation property by virtue of having stable rank one (see \cite[Proposition 6.5.1]{Bl98}). Hence, by Lemma \ref{lem:invertibles}, up to unitary isomorphism all invertible bimodules in $\Bim(\Z)$ are of the form $\Z_\theta$ for some $\theta \in \Aut(\Z)$. 

Since $(J,F)$ is a $\rm{C}^*$-tensor functor, we have that $F(\C_g)$ is invertible with inverse $F(\C_{g^{-1}})$ for each $g \in G$. 
We may therefore apply Proposition \ref{prop:tensor-functor} and deduce that there exists an $\bar{\omega}$-anomalous action of $G$ on $\Z$. By Theorem \ref{thm:JiangSu}, $[\bar{\omega}] = 0$ in $H^3(G, \T)$. Hence, $[\omega] = -[\bar{\omega}] = 0$.
\end{proof}
\begin{proof}[Proof of Corollary \ref{cor:UTC-UHF}]
Let $A = \bigotimes_{k \in \N}  \M_{n_k}$ be a UHF algebra with supernatural number $\mathfrak{n} = \prod_{k \in \N} n_k$. Let $G$ be a finite group and $\omega \in Z^3(G,\T)$.
Let $(F,J):\Hilb_\C(G, \omega) \rightarrow \Bim(A)$ be $\rm{C}^*$-tensor functor.

The ordered group $K_0(A)$ is isomorphic to the subgroup $Q(\mathfrak{n}) \subseteq \mathbb{Q}$ generated by $\{\tfrac{1}{n}: n \in \N, n \, | \, \mathfrak{n}\}$ with the order inherited from $\mathbb{Q}$. Since $\mathbb{Q}$ is uniquely divisible, any automorphism of $Q(\mathfrak{n})$ is determined by the image of $1$. It follows that the only automorphism of $Q(\mathfrak{n})$ with finite order is multiplication by -1, which doesn't preserve the order structure. Hence, $\Aut(K_0(A), K_0(A)_+)$ has no non-trivial elements of finite order. Since UHF algebras are simple, unital and have the cancellation property, we may use Lemma \ref{lem:invertibles} to deduce that, up to unitary isomorphism, all invertible bimodules of finite order in $\Bim(A)$ are of the form $A_\theta$ for some $\theta \in \Aut(A)$.

Since $(J,F)$ is a $\rm{C}^*$-tensor functor, we have that $F(\C_g)$ is invertible with inverse $F(\C_{g^{-1}})$ for each $g \in G$. Moreover, as $G$ is finite, $F(\C_g)$ has finite order for each $g \in G$. 
We may therefore apply Proposition \ref{prop:tensor-functor} and deduce that there exists an $\bar{\omega}$-anomalous action of $G$ on $A$. Since $[\bar{\omega}] = -[\omega]$ in $H^3(G,\T)$, they have the same order. Therefore, by Theorem \ref{thm:UHF}, if $r$ is the order of $[\omega]$ then $r^\infty$ divides the supernatural number $\mathfrak{n}$.
\end{proof}


\begin{thebibliography}{10}

\bibitem{SZBA17}
S.~Barlak, G.~Szabo.
\newblock Rokhlin actions of finite groups on {UHF}-absorbing {C$^*$}-algebras.
\newblock {\em Trans. Amer. Math. Soc.} 369(2):833--859, 2017.

\bibitem{Bl98}
B.~Blackadar.
\newblock {\em {$K$}-theory for operator algebras}, volume~5 of {\em
  Mathematical sciences research institute publications}.
\newblock Cambridge University Press, Cambridge, second edition, 1998.

\bibitem{BR67}
M.~Broise.
\newblock Commutateurs dans le groupe unitaire d’un facteur.
\newblock {\em J. Math. Pures Appl.}, 46(3):299--312, 1967.

\bibitem{BR12}
K.~S. Brown.
\newblock {\em Cohomology of groups}, volume~87 of {\em Graduate texts in
  mathematics}.
\newblock Springer Science and Business Media, 2012.

\bibitem{BRGRRI77}
L.~G. Brown, P.~Green, and M.~A. Rieffel.
\newblock Stable isomorphism and strong {M}orita equivalence of
  {C$^*$}-algebras.
\newblock {\em Pac. J. Math.}, 71(2):349--363, 1977.

\bibitem{BS70}
R.~C. Busby and H.~A. Smith.
\newblock Representations of twisted group algebras.
\newblock {\em Trans. Amer. Math. Soc.}, 149:503--537, 1970.

\bibitem{BMZ13}
A.~Buss, R.~Meyer, and C.~Zhu.
\newblock A higher category approach to twisted actions on {C$^*$}-algebras.
\newblock {\em Proc. Edinb. Math. Soc. (2)}, 56(2):387--426, 2013.

\bibitem{CO75}
A.~Connes.
\newblock Outer conjugacy classes of automorphisms of factors.
\newblock {\em Ann. Sci. \'{E}cole Norm. Sup. (4)}, 8(3):383--419, 1975.

\bibitem{CO76}
A.~Connes.
\newblock Classification of injective factors cases {II$_1$}, {II$_\infty$},
  {III$_\lambda$}, {$\lambda\ne 1$}.
\newblock {\em Ann. of Math. (2)}, pages 73--115, 1976.

\bibitem{CO77}
A.~Connes.
\newblock Periodic automorphisms of the hyperfinite factor of type {II$_1$}.
\newblock {\em Acta Sci. Math.}, 39:39--66, 1977.

\bibitem{CUPE79}
J.~Cuntz and G.~K. Pedersen.
\newblock Equivalence and traces on {C$^*$}-algebras.
\newblock {\em J. Funct. Anal.}, 33(2):135--164, 1979.

\bibitem{Dav96}
K.~R. Davidson.
\newblock {\em {C$^*$}-algebras by example}, volume~6 of {\em Fields institute
  monographs}.
\newblock American Mathematical Society, Providence, RI, 1996.

\bibitem{dlH13}
P.~de~la Harpe.
\newblock Fuglede-{K}adison determinant: theme and variations.
\newblock {\em Proc. Natl. Acad. Sci. U.S.A.}, 110(40):15864--15877, 2013.

\bibitem{SdlH84}
P.~de~la Harpe and G.~Skandalis.
\newblock D\'{e}terminant associ\'{e} \`a une trace sur une alg\'{e}bre de
  {B}anach.
\newblock {\em Ann. Inst. Fourier (Grenoble)}, 34(1):241--260, 1984.

\bibitem{SdlH84a}
P.~de~la Harpe and G.~Skandalis.
\newblock Produits finis de commutateurs dans les {C$^*$}-alg\`ebres.
\newblock {\em Ann. Inst. Fourier (Grenoble)}, 34(4):169--202, 1984.

\bibitem{EKQR06}
S.~Echterhoff, S.~Kaliszewski, J.~Quigg, and I.~Raeburn.
\newblock A categorical approach to imprimitivity theorems for
  {C$^*$}-dynamical systems.
\newblock {\em Mem. Amer. Math. Soc.}, 180, 2006.

\bibitem{EILMAC47}
S.~Eilenberg and S.~MacLane.
\newblock Cohomology theory in abstract groups. {II}: {G}roup extensions with a
  non-abelian kernel.
\newblock {\em Ann. of Math. (2)}, pages 326--341, 1947.

\bibitem{EGLN15}
G.~A. Elliott, G.~Gong, H.~Lin, and Z.~Niu.
\newblock On the classification of simple amenable {C$^*$}-algebras with finite
  decomposition rank, {II}.
\newblock arXiv:1507.03437, 2015.

\bibitem{EGNO}
P.~Etingof, S.~Gelaki, D.~Nikshych, and V.~Ostrik.
\newblock {\em Tensor categories}, volume 205 of {\em Mathematical surveys and
  monographs}.
\newblock American Mathematical Society, Providence, RI, 2016.

\bibitem{KAEV98}
D.~E. Evans and Y.~Kawahigashi.
\newblock {\em Quantum symmetries on operator algebras}, volume 147 of {\em
  Oxford mathematical monographs}.
\newblock Clarendon Press Oxford, 1998.

\bibitem{GLR85}
P.~Ghez, R.~Lima, and J.~E. Roberts.
\newblock {W$^*$}-categories.
\newblock {\em Pac. J. Math.}, 120(1):79--109, 1985.

\bibitem{GlimmUHF}
J.~G. Glimm.
\newblock On a certain class of operator algebras.
\newblock {\em Trans. Amer. Math. Soc.}, 95:318--340, 1960.

\bibitem{Gl05}
H.~Gl\"{o}ckner.
\newblock Fundamentals of direct limit {L}ie theory.
\newblock {\em Compos. Math.}, 141(6):1551--1577, 2005.

\bibitem{GLN15}
G.~Gong, H.~Lin, and Z.~Niu.
\newblock Classification of finite simple amenable $\mathcal{Z}$-stable
  {C$^*$}-algebras {I},{II}.
\newblock {\em C. R. Math. Acad. Sci. Soc. R. Can.}, 42(3-4):63--539, 2020.

\bibitem{PIS18}
P.~Grossman, M.~Izumi, and N.~Snyder.
\newblock The {A}saeda-{H}aagerup fusion categories.
\newblock {\em J. Reine Angew. Math.}, 743:261--305, 2018.

\bibitem{HAA14}
U.~Haagerup.
\newblock Quasitraces on exact {C$^*$}-algebras are traces.
\newblock {\em C. R. Math. Acad. Sci. Soc. R. Can.}, 36:67--92, 2014.

\bibitem{HIG88}
N.~Higson.
\newblock Algebraic {K}-theory of stable {C$^*$}-algebras.
\newblock {\em Adv. Math.}, 67(1):1--140, 1988.

\bibitem{HU14}
H.-L. Huang, G.~Liu, and Y.~Ye.
\newblock The braided monoidal structures on a class of linear {G}r-categories.
\newblock {\em Alg. and Rep. Theory}, 17(4):1249--1265, 2014.

\bibitem{IZU04}
M.~Izumi.
\newblock Finite group actions on {C}$^*$-algebras with the {R}ohlin property, {I}.
\newblock {\em Duke Math. Journal} 122(2):233--280, 2004.


\bibitem{Iz93}
M.~Izumi.
\newblock Subalgebras of infinite {C$^*$}-algebras with finite {W}atatani
  indices. {I}. {C}untz algebras.
\newblock {\em Comm. Math. Phys.}, 155(1):157--182, 1993.

\bibitem{Iz98}
M.~Izumi.
\newblock Subalgebras of infinite {C$^*$}-algebras with finite {W}atatani
  indices. {II}. {C}untz-{K}rieger algebras.
\newblock {\em Duke Math. J.}, 91(3):409--461, 1998.


\bibitem{Jac13}
B.~Jacelon.
\newblock A simple, monotracial, stably projectionless {C$^*$}-algebra.
\newblock {\em J. Lond. Math. Soc. (2)}, 87(2):365--383, 2013.

\bibitem{Jiang-Su}
X.~Jiang and H.~Su.
\newblock On a simple unital projectionless {C$^*$}-algebra.
\newblock {\em Amer. J. Math.}, pages 359--413, 1999.

\bibitem{JO20}
C.~Jones.
\newblock Remarks on anomalous symmetries of {C$^*$}-algebras.
\newblock {\em Comm. Math. Phys.}, 388(1):385--417, 2021.

\bibitem{JP16}
C.~Jones and D.~Penneys.
\newblock Operator algebras in rigid {$\rm C^*$}-tensor categories.
\newblock {\em Comm. Math. Phys.}, 355(3):1121--1188, 2017.

\bibitem{JOPE19}
C.~Jones and D.~Penneys.
\newblock Realizations of algebra objects and discrete subfactors.
\newblock {\em Adv. Math.}, 350:588--661, 2019.

\bibitem{JON79}
V.~F.~R. Jones.
\newblock An invariant for group actions.
\newblock In {\em Alg\`ebres d'op\'{e}rateurs ({S}\'{e}m., {L}es
  {P}lans-sur-{B}ex, 1978)}, volume 725 of {\em Lecture Notes in Mathematics},
  pages 237--253. Springer, Berlin, 1979.

\bibitem{JON80}
V.~F.~R. Jones.
\newblock Actions of finite groups on the hyperfinite type {${\rm {II}}_{1}$}
  factor.
\newblock {\em Mem. Amer. Math. Soc.}, 28(237):v+70, 1980.

\bibitem{KWP04}
T.~Kajiwara, C.~Pinzari, and Y.~Watatani.
\newblock Jones index theory for {H}ilbert {C$^*$}-bimodules and its
  equivalence with conjugation theory.
\newblock {\em J. Funct. Anal.}, 215(1):1--49, 2004.

\bibitem{KW00}
T.~Kajiwara and Y.~Watatani.
\newblock Jones index theory by {H}ilbert {C$^*$}-bimodules and {$K$}-theory.
\newblock {\em Trans. Amer. Math. Soc.}, 352(8):3429--3472, 2000.

\bibitem{Ki95}
E.~Kirchberg.
\newblock Exact {C$^*$}-algebras, tensor products, and the classification of
  purely infinite algebras.
\newblock In {\em Proc. of the {I}nter. {C}ongress of {M}ath., {V}ol.\ 1, 2
  ({Z}\"urich, 1994)}, pages 943--954. Birkh\"auser, Basel, 1995.
  
\bibitem{KI77}
A.~Kishimoto.
\newblock On the fixed point algebra of a {UHF} algebra under a periodic automorphism of product type.
\newblock {\em Publ. Res. Ins. Math. Sci.} 13(3):777--791, 1977.

\bibitem{Lance95}
E.~C. Lance.
\newblock {\em Hilbert {C$^*$}-modules}, volume 210 of {\em London mathematical
  society lecture note series}.
\newblock Cambridge University Press, Cambridge, 1995.

\bibitem{Lin07}
H.~Lin.
\newblock Simple nuclear {C$^*$}-algebras of tracial topological rank one.
\newblock {\em J. Funct. Anal.}, 251(2):601--679, 2007.

\bibitem{LORO97}
R.~Longo and J.~Roberts.
\newblock A theory of dimension.
\newblock {\em K-Theory}, 11:103--159, 1997.

\bibitem{MAC13}
S.~MacLane.
\newblock {\em Categories for the working mathematician}, volume~5 of {\em
  Graduate texts in mathematics}.
\newblock Springer Science and Business Media, 2013.

\bibitem{Mug03}
M.~M\"{u}ger.
\newblock From subfactors to categories and topology. {I}. {F}robenius algebras
  in and {M}orita equivalence of tensor categories.
\newblock {\em J. Pure Appl. Algebra}, 180(1):81--157, 2003.

\bibitem{MvN36}
F.~J. Murray and J.~von Neumann.
\newblock On rings of operators.
\newblock {\em Ann. of Math. (2)}, 37(1):116--229, 1936.

\bibitem{NATA59}
M.~Nakamura, Z.~Takeda.
\newblock On the extensions of finite factors, {I}.
\newblock {\em Proc. Japan Acad.}, 35(4):149--154, 1959.


\bibitem{KernelDet}
P.~W. Ng and L.~Robert.
\newblock The kernel of the determinant map on pure {C$^*$}-algebras.
\newblock {\em Houston J. Math.}, 43(1):139--168, 2017.

\bibitem{OC88}
A.~Ocneanu.
\newblock Quantized groups, string algebras and {G}alois theory for algebras.
\newblock In {\em Operator algebras and applications, {V}ol. 2}, volume 136 of
  {\em London Mathematical Society Lecture Note Series}, pages 119--172.
  Cambridge University Press, Cambridge, 1988.

\bibitem{OC06}
A.~Ocneanu.
\newblock {\em Actions of discrete amenable groups on von {N}eumann algebras},
  volume 1138 of {\em Lecture notes in mathematics}.
\newblock Springer, 1985.

\bibitem{OSTE14}
H.~Osaka, T.~Teruya.
\newblock Strongly self-absorbing property for inclusions of {C$^*$}-algebras with a finite {W}atatani index.
\newblock {\em Trans. Amer, Math. Soc.} 366(3):1685--1702, 2014.


\bibitem{PR89}
J.~A. Packer and I.~Raeburn.
\newblock Twisted crossed products of {C$^*$}-algebras.
\newblock {\em Math. Proc. Camb. Phil. Soc.}, 106(2):293--311, 1989.

\bibitem{Phi00}
N.~C. Phillips.
\newblock A classification theorem for nuclear purely infinite simple
  {C$^*$}-algebras.
\newblock {\em Doc. Math.}, 5(49):114, 2000.

\bibitem{PO94}
S.~Popa.
\newblock Classification of amenable subfactors of type {II}.
\newblock {\em Acta Math.}, 172(2):163--255, 1994.

\bibitem{RO04}
M.~R{\o}rdam.
\newblock The stable and the real rank of $\mathcal{Z}$-absorbing
  {C$^*$}-algebras.
\newblock {\em Int. J. Math.}, 15(10):1065--1084, 2004.

\bibitem{IntroK}
M.~R{\o}rdam, F.~Larsen, and N.~Laustsen.
\newblock {\em An introduction to K-theory for {C$^*$}-algebras}, volume~49 of
  {\em London mathematical society student texts}.
\newblock Cambridge University Press, 2000.

\bibitem{Su80}
C.~E. Sutherland.
\newblock Cohomology and extensions of von {N}eumann algebras. {I}, {II}.
\newblock {\em Publ. Res. Inst. Math. Sci.}, 16(1):105--133, 135--174, 1980.

\bibitem{Ta59}
Z.~Takeda.
\newblock On the extensions of finite factors. {II}.
\newblock {\em Proc. Japan Acad.}, 35:215--220, 1959.

\bibitem{TSH98}
N.~Tatsuuma, H.~Shimomura, and T.~Hirai.
\newblock On group topologies and unitary representations of inductive limits
  of topological groups and the case of the group of diffeomorphisms.
\newblock {\em J. Math. Kyoto Univ.}, 38(3):551--578, 1998.

\bibitem{Th95}
K.~Thomsen.
\newblock Traces, unitary characters and crossed products by {$\mathbb{Z}$}.
\newblock {\em Publ. Res. Inst. Math. Sci.}, 31(6):1011--1029, 1995.

\bibitem{TWW17}
A.~Tikuisis, S.~White, and W.~Winter.
\newblock Quasidiagonality of nuclear {C$^*$}-algebras.
\newblock {\em Ann. of Math. (2)}, 185(1):229--284, 2017.

\bibitem{TD08}
T.~tom Dieck.
\newblock {\em Algebraic topology}, volume~8 of {\em European mathematical
  society textbooks in mathematics}.
\newblock European Mathematical Society, 2008.

\bibitem{Wei94}
C.~A. Weibel.
\newblock {\em An introduction to homological algebra}, volume~38 of {\em
  Cambridge studies in advanced mathematics}.
\newblock Cambridge University Press, Cambridge, 1994.

\bibitem{Wi12}
W.~Winter.
\newblock Nuclear dimension and {$\mathcal{Z}$}-stability of pure
  {C$^*$}-algebras.
\newblock {\em Invent. Math.}, 187(2):259--342, 2012.

\end{thebibliography}
\end{document}